\documentclass[reqno]{amsart}
\usepackage{amssymb}
\usepackage{hyperref}
\usepackage{mathrsfs}

\usepackage[margin=1in]{geometry}  
\usepackage{graphicx}              
\usepackage{amsmath}               
\usepackage{amsfonts}              
\usepackage{amsthm}                
\usepackage{amssymb}

\usepackage{color}
\numberwithin{equation}{section}
\newtheorem{theorem}{Theorem}[section]
\newtheorem{lemma}[theorem]{Lemma} 
\newtheorem{proposition}[theorem]{Proposition}
\newtheorem{definition}[theorem]{Definition}
\newtheorem{remark}[theorem]{Remark}

\allowdisplaybreaks

\numberwithin{equation}{section}

\usepackage{graphicx}              
\graphicspath{{./img/}}	

\begin{document}

\title[On asymptotic behavior for a class of diffusion]{On asymptotic behavior for a class of diffusion equations involving the fractional $\wp(\cdot)-$Laplacian as $\wp(\cdot)$ goes to $\infty$ }

\author{Lauren Maria Mezzomo Bonaldo and Elard Juarez Hurtado}
\thanks{Department of Mathematics Universidade Federal do Espírito Santo,  Espírito Santo ES, Brazil.
E-mail: lauren.bonaldo@ufes.br}\, 
\thanks{Department of Mathematics, Universidade  de Brasília,  Brasília DF, Brazil.
E-mail: elardjh2@gmail.com Supported
by CAPES}

\date{}
\maketitle
 \begin{abstract}

In this manuscript, we will study the asymptotic behavior for a class of nonlocal diffusion equations associated with the weighted fractional $\wp(\cdot)-$Laplacian operator involving constant/variable exponent. In the case of constant exponents, under some appropriate conditions, we will  study the existence of solutions and asymptotic behavior of solutions by employing the subdifferential approach and we will study the problem when $\wp$ goes to $\infty$.    Already, for case the weighted fractional $\wp(\cdot)$-Laplacian operator, we will also study the asymptotic behavior of the problem solution when $\wp(\cdot)$ goes to $\infty$, in the whole or in a subset of the domain (the problem involving the fractional $\wp(\cdot)$-Laplacian presents a discontinuous exponent).  
To obtain the results of the asymptotic behavior in both problems it will be via Mosco convergence.  
 \end{abstract}

\section{Introduction}\label{intro}\quad
\noindent
 We investigate the existence of solutions and analyze the asymptotic behavior 
of solutions of a class of nonlocal diffusion equations when  $\wp(\cdot)$ goes to $\infty$ associated with the  weighted fractional  $\wp(\cdot)$-Laplacian operator  defined by
\begin{equation}\label{operator}
(-\Delta)_{\wp(\cdot), \mathtt{A}}^{s}=P.V.\displaystyle{\int_{\Omega}\mathtt{A}(x,y,t)\frac{|u(x)-u(y)|^{\wp(x,y)-2}(u(x)-u(y))\,dy}{|x-y|^{N+s\wp(x,y)}},\,\,x\in \Omega},
\end{equation}

\noindent where $\Omega\subset \mathbb{R}^{N}$ be a bounded domain with smooth boundary $\partial \Omega,$ \text{P.V.} is the Cauchy principal value, $s\in (0,1),$  $\mathtt{A}:\Omega\times \Omega\times (0,T)\to \mathbb{R}$ is a weighted function and $\wp:\overline{\Omega}\times \overline{\Omega}\to (1,\infty)$ both functions verifying adequate conditions. In the case of the weighted function $\mathtt{A}(\cdot,\cdot,t)\equiv \mathtt{A}(\cdot,\cdot)$ and when $\mathtt{A}(\cdot,\cdot,\cdot)\equiv 1$ have been treated in the papers \cite{palatucci,pucci,scott} and the references therein.
In the case when $\mathtt{A}(\cdot,\cdot,\cdot)\equiv 1$, the weighted fractional $\wp(\cdot)$-Laplacian operator is the fractional $\wp(\cdot)-$Laplacian operator, that is, the operator $(-\Delta)_{\wp(\cdot)}^{s},$ which is a fractional version of the $\wp(\cdot)$-Laplacian operator, given by $\operatorname{div}(|D u|^{\wp(x)-2}D u),$  these operators play a fundamental role in various applications: biology, data science,  geophysics, tomographic reconstruction, novel exterior optimal control, population dynamics, we refer e.g.  \cite{antil,antil2,antil4,bisch,xavier1,xavier2,caffarelli2,valdinoci,galiano,jin,tai,antil3} and references therein.




We analyze this paper in two situations distinct: 

First,  we will study for a class of nonlocal diffusion equations associated with the weighted fractional $\wp-$Laplacian operator, we will specifically study the existence of solutions and asymptotic behavior of the solutions of a class of nonlocal diffusion equation associated with the weighted fractional $\wp$-Laplacian when $\wp$ goes to $\infty$,  precisely,  we will study the following problem 
\begin{equation}\tag{$\mathscr{C}_{\varphi^{t}_{\wp},f,u_{0}}$}\label{2.13} 
\left\{
\begin{array}{rcl}
&&\partial_t u + (-\Delta)^{s}_{\wp, \mathtt{A}} u  = f(x,t) \;\; \mbox{in} \;\; \Omega \times (0,T),\\
&& u  = 0 \;\; \mbox{on} \;\; \partial\Omega \times (0,T),\,\,u (\cdot, 0)  = u_{0} \;\; \mbox{in} \;\; \Omega
\end{array} 
\right.
\end{equation}
and under appropriate conditions, we will study the well-posedness of solutions for the following periodic problem    
\begin{equation}\tag{$\mathcal{P}_{\varphi^{t}_{\wp},f}$}\label{2.18} 
\left\{
\begin{array}{rcl}
&&\partial_t u + (-\Delta)^{s}_{\wp,\mathtt{A}} u  = f(x,t) \;\; \mbox{in} \;\; \Omega \times (0,T),\\
&& u  = 0 \;\; \mbox{on} \;\; \partial\Omega \times (0,T),\,\,u(x, 0)  = u_{0}(T) \;\; \mbox{in} \;\; \Omega,
\end{array} 
\right.
\end{equation}
  where $\partial_{t}u=\partial u/\partial t$ denote the partial derivative of $u,$ $f$ and $u_{0}$ are functions given satisfying certain conditions.

  

In the second situation, we will consider the weighted $\mathtt{A}(\cdot,\cdot,\cdot)\equiv 1$ and we will analyze the asymptotic behavior as $\wp(\cdot)$ goes to $\infty$ for the problem
\begin{equation}\tag{$\mathcal{P}_{\wp}$}\label{IP}
\left\{
\begin{array}{rcl}
&&\partial_t u + (-\Delta)^{s}_{\wp(\cdot)} u  = f(x,t) \;\; \mbox{in} \;\; \Omega \times (0,T),\\
&&u  = 0 \;\; \mbox{on} \;\; \partial\Omega \times (0,T),\,\,u (\cdot, 0)  = u_{0} \;\; \mbox{in} \;\; \Omega,
\end{array} 
\right.
\end{equation}
where $f:\Omega\times (0,\infty)\to \mathbb{R}$ and $u_{0}:\Omega\to \mathbb{R}$ functions are given satisfying certain conditions. 

More precisely, we study the asymptotic behavior of problem \eqref{IP}, by substituting $\wp(\cdot,\cdot)$ instead of $\wp_{\jmath}(\cdot,\cdot)$, that is,  we shall investigate the limiting behavior as $\wp_{\jmath}(\cdot,\cdot) \rightarrow \infty$ of the solutions $u_{\wp_{\jmath}(\cdot)}=u_{\wp_{\jmath}(x,y)}(x,t)$ for  the following problem
\begin{equation}\label{aproximado}
\left\{
\begin{array}{rcl}
&&\partial_t u_{\wp_{\jmath}(\cdot)} + (-\Delta)^{s}_{\wp_{\jmath}(\cdot)} u_{\wp_{\jmath}(\cdot)}  = f_{\wp_{\jmath}(\cdot)}(x,t)  \;\; \mbox{in} \;\; \Omega \times (0,T),\\
&&u_{\wp_{\jmath}(\cdot)}  = 0 \;\; \mbox{on} \;\; \partial\Omega \times (0,T),\,\,u_{\wp_{\jmath}(\cdot)} (\cdot, 0)  = u_{\wp_{\jmath}(\cdot)}^{0} \;\; \mbox{in} \;\; \Omega,
\end{array} 
\right.
\end{equation}
 where, for all $\jmath\in \mathbb{N}$,  the sequence $\wp_{\jmath}:\overline{\Omega}\times \overline{\Omega}\to (1,\infty)$ is such that $\wp_{\jmath}\in \mathscr{C}(\overline{\Omega}\times \overline{\Omega})$ for  all $(x,y)\in\overline{\Omega}\times \overline{\Omega}$ and   $f_{\wp_{\jmath}(\cdot)}$ and $u_{\wp_{\jmath}(\cdot)}^{0}$  are functions given satisfying certain conditions.
Fixed $\jmath,$ in the problem \eqref{IP}, we have the existence of a unique solution $u_{\wp_{\jmath}(\cdot)}$ of the corresponding   problem  \eqref{aproximado} (see for instance \cite{elardjh}).


Regarding the first situation, we will denote the limiting solution by $u$ and show that this solution verifies a variational characterization in the set

\begin{equation*}
\mathtt{K}^{t} := \left\lbrace u \in W^{s,2}_{0}(\Omega): \sup_{ (x,y) \in \Omega \times \Omega,\,\, x\neq y  } \dfrac{|v(x) - v(y)|}{|x-y|^{s}} \leqslant \frac{1}{\mathtt{A}(x,y,t)} \right\rbrace
\end{equation*}
where the weighted $\mathtt{A}$ verify the hypothesis $(W_{1}).$ The set $\mathtt{K}^{t}$ will play a very important role in studying the asymptotic behavior for this problem.

In the second situation,   we will study the behavior 
of  sequence of solutions $u_{\wp_{j}(\cdot)}$ of \eqref{aproximado} when we consider that the sequence of functions $\wp_{j}(\cdot,\cdot)$ goes to $\infty$ in the domain $\Omega\times \Omega$. In this situation, we denote this limit solution by $\mathfrak{u}_{\infty}$ and we will prove that this solution satisfies a variational characterization in the set 
\begin{equation*}
\mathbb{K}_{\infty} = \left\lbrace v\in W^{s,2}_{0}(\Omega) : \sup_{ (x,y) \in \Omega \times \Omega,\,\, x\neq y  } \dfrac{|v(x) - v(y)|}{|x-y|^{s}} \leqslant 1 \right\rbrace,
\end{equation*}

\noindent the set $\mathbb{K}_{\infty}$ plays a crucial role throughout the entire work. 
After, we study the problem \eqref{aproximado} when $ \wp_{\jmath}(\cdot,\cdot)\to \infty $ in a subdomain of $\Omega\times \Omega$. More specifically, we consider the following condition

\begin{equation*}\label{sb1}
\wp_{\jmath}(x,y) =
\begin{cases}
\kappa_{\jmath}(x,y) \to \infty \mbox{ if } (x,y) \in \mathcal{O}\times \mathcal{O}, \\
\kappa(x,y) < \infty \mbox{ if }(x,y) \in (\Omega \times \Omega) \setminus \overline{\mathcal{O}\times \mathcal{O}} 
\end{cases}
\end{equation*}

\noindent as $\jmath \to \infty,$ where $\mathcal{O}$ is a nonempty open subset of $\Omega$ and continuous functions  $ \kappa (\cdot,\cdot),$ $\kappa_{\jmath}(\cdot,\cdot) $ with values in $(1,\infty)$. In this situation, the limit problem is described as a mixture of two problems, a nonlinear diffusion equation involving the $\wp(\cdot)$-Laplacian in $(\Omega \times \Omega) \setminus \overline{\mathcal{O}\times \mathcal{O}}$ and a quasivariational inequality of evolution over $\mathcal{O}\times \mathcal{O}$. In contrast to the preceding case, the constraint set of quasivariational inequality depends on the unknown function. As far as we know, these results are new and the techniques developed are not standard at all.



 The problems that we study in this manuscript, namely the problems \eqref{2.13} - \eqref{IP}, are generalized, improve, and complement some results obtained in some recent work \cite{akagip,akagi,elardjh}, in the sense that we study problems involving a fractional operator with variable exponents, we also study qualitative properties for this class of problems. More precisely, when $s\geqslant 1$ and $\wp(\cdot)\equiv(\mbox{constant/variable})$ stationary and evolution problems were studied by \cite{akagip,akagi,mazon22,mazon33,attouch2,manfredi1,lindgren,palatucci,ken1}.


The motivation of the present investigation contains several aspects. The first is that in general nonlocal diffusion problems are important for modeling evolutionary problems, for example, problems in physics, biology, engineering, as mentioned in the beginning. The second interesting aspect of this work is regarding the non-standard growth of the variable exponents. In the literature, most of the known results refer to the stationary case with the growth condition $\wp(x)$ (see, for instance \cite{elardlaurenarx,lauren} and the references therein). Also, we want to refer to \cite{galiano,gilboa,jin,tai,antil3} for instance, for some numerical aspects related to
the numerical approximation of problems related to the parabolic fractional $p-$Laplacian.

Problems involving the fractional weighted $\wp(\cdot)-$Laplacian with variable exponents and weights have been very little research for evolution equations so far, it should also be noted that the asymptotic behavior of this type of problem has not been studied when $\wp(\cdot)$ goes to infinity. The appearance of variable exponents in the structure of this problem produce many difficulties and challenges when applying the techniques and strategies of key tools necessary, important for establishing our results, this structure also generates serious difficulties in obtaining estimates in order to take the limit as $ \wp(\cdot,\cdot)\to \infty $ once the canonical techniques cannot be used to obtain estimations as in the case of local problems. However, as mentioned at the beginning, partial differential equations involving functional spaces with variable exponents have increasingly grown interested and motivated in recent years.

To explain the motivation of our work for example we refer to the article \cite{warma}, where the authors, for $\wp(\cdot,\cdot)\equiv \wp (\mbox{constant}),$ they study the following problem

\begin{equation*} 
\left\{
\begin{array}{rcl}
&&\partial_t u(x,t) + (-\Delta)^{s}_{\wp} u(x,t)-|u(x,t)|^{q-2}u(x,t)  = f(x,t) \;\; \mbox{in} \;\; \Omega \times (0,T),\\
&& u(x,t) = 0 \;\; \mbox{on} \;\; (\mathbb{R}^{N}\setminus \Omega) \times (0,T),\\
&& u(x, 0)  = u_{0}(x) \;\; \mbox{in} \;\; \Omega,
\end{array} 
\right.
\end{equation*}
they show the existence of locally defined solutions to the problem with any initial condition $u_{0}\in L^{r}(\Omega)$ where $r\geqslant 2$ verifies $r>N(q-p)/sp,$ also show that finite time explosion is possible for these problems.


In \cite{giacomoni}, the authors studied the problem

\begin{equation*} 
\left\{
\begin{array}{rcl}
&&\partial_t u(x,t) + (-\Delta)^{s}_{\wp} u(x,t)+g(x,u)  = f(x,u) \;\; \mbox{in} \;\; \Omega \times (0,T),\\
&& u(x,t) = 0 \;\; \mbox{on} \;\; (\mathbb{R}^{N}\setminus \Omega) \times (0,T),\\
&& u(x, 0)  = u_{0}(x) \;\; \mbox{in} \;\; \Omega,
\end{array} 
\right.
\end{equation*}
they investigate the asymptotic behavior of global weak solutions, specifically, they show that suitable conditions on $f$ and $g$, that global solutions converge to the only stationary solution when $t\to \infty.$

 In \cite{mazon11}, Maz\'{o}n et al., establish existence, some properties as extinction regularity of the weak solution and the entropy solution with general data to the problem \eqref{IP} for the constant exponent $\wp$ in different situations. Besides, when $\wp(\cdot)\equiv \wp (\mbox{constant}),$ V\'{a}zquez \cite{vazquez} studied the existence, uniqueness and properties of strong nonnegative solutions of the following equation

$$\partial_{t}u(x,t)+\int_{\mathbb{R}^{N}}\frac{\mathfrak{a}(u(x,t)-u(y,t))}{|x-y|^{ N+s\wp}}\,dy=0,$$
where $\mathfrak{a}(z)=|z|^{\wp-2}z,$ for $1<\wp<\infty,$ and $0<s<1.$

 Moreover, in the constant exponent case $\wp(\cdot,\cdot)\equiv \wp,$ the problem of the diffusion limit is connected with the problem
  optimal transport through a nonlocal version of a sandpile growth model (see e.g. \cite{mazon22,mazon33,mazon11}). Furthermore, these problems appear in many applications, as for example in continuum mechanics, phase transition phenomena, image process, game theory, and L\'evy processes (see e.g. \cite{cafarelli,galiano,gilboa,laskin}).
 
  It should also be emphasized that when $\wp(\cdot,\cdot) \equiv \wp,$ the limit problem is related to the infinity fractional Laplacian:
$$(\mathfrak{L}_{\infty}^{s}u)(x)=(\mathfrak{L}_{\infty}^{s}u)_{+}(x)+(\mathfrak{L}_{\infty}^{s}u)_{-}(x),$$
where
\begin{equation*}
(\mathfrak{L}_{\infty}^{s}u)_{+}(x)=\sup_{ y\in \overline{\Omega},\,\,y\neq x}\frac{u(y)-u(x)}{|y-x|^{s}} \,\,\, ,(\mathfrak{L}_{\infty}^{s}u)_{-}(x)=\inf_{ y\in \overline{\Omega},\,\,y\neq x}\frac{u(y)-u(x)}{|y-x|^{s}}\mbox{ for } x\in \Omega,
\end{equation*}
for more details see e.g. \cite{lindgren,mayte}. The infinity Laplacian/infinity fractional Laplacian equation is currently applied in various fields of mathematics, for instance, in computer vision, surface reconstruction, optimal transport, see e.g. \cite{aronson1,aronson2,manfredi1,crandal,crandal2,kawohl,manfredi2,manfredirossi,rossi1,rossi2,linds,vazquez}.

Recently in \cite{elardjh}, Hurtado studied well-posedness of the problem \eqref{IP} governed by subdifferential operators in the framework of semigroups generated by maximal monotone operators, the asymptotic behavior of \eqref{IP} are also studied when $t\to \infty$ and the existence of global attractors.

Bearing in mind the aforementioned investigations, a natural question is what happens with the fractional $\wp(\cdot)-$Laplacian when $\wp(\cdot)$ goes to $\infty$. Thus, in this paper, we focus our attention on the study of solutions of the nonlocal diffusion problems when $\wp_{\jmath}(\cdot,\cdot)\to \infty,$ taking of inspiration the works \cite{akagi,mazon22,mazon33,elardjh,mazon11}. In order to prove the main results, the notion of Mosco convergence will be used which provides an appropriate environment to study asymptotic behavior for a wide class of variational problems in mathematical analysis, see for instance \cite{akagip,attouch,mayte2,palatucci,warma,manfredirossi,mosco,pucci}.




The main contributions of this paper are the following:
\begin{itemize}
\item[$\bullet$] As far as we know,  this paper is the first in the literature to work with a fractional weighted $ \wp(\cdot)-$Laplacian and study the existence of a solution and asymptotic behavior for a parabolic problem when $ \wp(\cdot)$ goes to infinity.
\item[$\bullet$] Little is currently known about parabolic problems associated with the $\wp(\cdot)-$ Laplacian fractional operator, this work is the first attempt to analyze a nonlocal problem where the variable exponent is discontinuous somewhere in the domain, we assume that $\wp(\cdot,\cdot)$ is infinity in a subdomain and finite in the complement.
\end{itemize}

\section{Notation, assumptions and statement of the main result}

  Throughout the text we use the notation, if $0<s<1,$ $\Omega\times \Omega=\Omega^{2},$ $\wp:\overline{\Omega}^{2}\to  (1,\infty)$ and $\wp'(\cdot,\cdot):=\displaystyle\frac{\wp(\cdot,\cdot)}{\wp(\cdot,\cdot)-1},$ then the space $\mathcal{W}^{-s,\wp'(\cdot,\cdot)}(\Omega)$ is defined as
usual to be the dual of the reflexive Banach space $\mathcal{W}_{0}^{s,\wp(\cdot,\cdot)}(\Omega):=\mathbb{X},$ that is $\mathcal{W}^{-s,\wp'(\cdot,\cdot)}(\Omega):=(\mathcal{W}_{0}^{s,\wp(\cdot,\cdot)}(\Omega))'=\mathbb{X}'.$ For $u\in \mathcal{W}^{s,\wp(\cdot,\cdot)}(\Omega)$ we denote by $\mathcal{G}_{(s,\wp(\cdot,\cdot))}u(x,y)$ the function defined in $\Omega^{2}$ by
\begin{equation*}\label{grad}
\mathcal{G}_{(s,\wp(\cdot,\cdot))}u(x,y):=\frac{u(x)-u(y)}{|x-y|^{\frac{N}{\wp(\cdot,\cdot)}+s}}.
\end{equation*}
We will also denote by $\mathcal{G}_{s}u(x,y):=\frac{u(x)-u(y)}{|x-y|^{s}}$ the H\"{o}lder quotient of order $s$ and the measure on $\Omega^{2}$ given by $d\mu:=\frac{dx\,dy}{|x-y|^{N}}$ and we denote $u_{\wp_{\jmath(\cdot,\cdot)}}:=u_{\wp_{\jmath(\cdot)}}$.

For $s\in (0,1)$ and $\wp=\infty,$ we let
\begin{equation*}
W^{s,\infty}(\Omega):=\bigg\{u\in L^{\infty}(\Omega):\sup_{x,y\in \Omega,\,\,x\neq y}\frac{|u(x)-u(y)|}{|x-y|^{s}}<\infty   \bigg\}.
\end{equation*}
It follows that for $\displaystyle{wp^{-}:=\min_{(x,y) \in \overline{\Omega^{2}}} \wp(x,y)\geqslant \max\left\{\displaystyle\frac{2N}{N+2s},1\right\}}$ and $s\in (0,1),$ we have 
$$\mathcal{W}_{0}^{s,\wp(\cdot,\cdot)}(\Omega)\hookrightarrow L^{2}(\Omega)\cong (L^{2}(\Omega))' \hookrightarrow (\mathcal{W}_{0}^{s,\wp(\cdot,\cdot)}(\Omega))' $$
with continuous and dense embeddings, called Gelfand triple.

The definition of solution of the equation 
\begin{equation}\label{pb}
\partial_t u + (-\Delta)^{s}_{\wp, \mathtt{A}} u  = f(x,t)
\end{equation}
 is as follows.
\begin{definition}\label{31}
 A function $u\in \mathscr{C}([0,T];L^{2}(\Omega))$ is said to be strong solution of \eqref{pb} if the following conditions are satisfied:
 \begin{itemize}
 \item[(i)] $u(\cdot,t)$ is an $L^{2}(\Omega)-$valued absolutely continuous function on $[0,T];$
 \item[(ii)] $u(\cdot,t)\in W_{0}^{s,\wp}(\Omega)$ for a.e. $t\in (0,T)$ and 
 \begin{equation*}
\begin{split} 
 \int_{\Omega}\partial_{t}u(x,t)\,dx+\int_{\Omega^{2}}\mathtt{A}(x,y,t)\frac{|u(x)-u(y)|^{\wp-2}(u(x)-u(y))(v(x)-v(y))}{|x-y|^{N+s\wp}}\,dx \,dy\\
 =\int_{\Omega}f(x,t)v(x)\,dx
 \end{split}
 \end{equation*}
for all $v\in W_{0}^{s,\wp}(\Omega)$ and a.e. $t\in (0,T).$ 
 \end{itemize}
\end{definition}

Furthermore, we say that $u\in \mathscr{C}([0,T];L^{2}(\Omega))$ is a weak solution 
of \eqref{pb} if there exist sequence $(f_{\wp_{j}})_{j \in \mathbb{N}}\subset L^{1}(0,T;L^{2}(\Omega))$ and $(u_{j})_{j \in \mathbb{N}}\subset \mathscr{C}([0,T];L^{2}(\Omega))$ such that $u_{j}$ is a strong solution of \eqref{pb}, $f_{\wp_{j}}\to f_{j}$ strongly in $L^{1}(0,T;L^{2}(\Omega))$ and $u_{j}\to u$ in $\mathscr{C}([0,T];L^{2}(\Omega))$ as $j\to \infty.$



The definition of solution of the problem \eqref{IP} is as follows:

\begin{definition}
A function $u\in \mathscr{C}(\mathbb{R}_{0}^{+}; L^{2}(\Omega))$ is said to be a solution of \eqref{IP}, if the following conditions are all satisfied:
\begin{itemize}
\item[$\textbf{(1)}$] $u\in W_{loc}^{1,2}(\mathbb{R}_{0}^{+};L^{2}(\Omega))\cap \mathscr{C}_{w}(\mathbb{R}_{0}^{+};\mathbb{X}),$
\item[$\textbf{(2)}$] $u(\cdot,0)=u_{0}$ a.a. in $\Omega,$
\item[$\textbf{(3)}$] For all $\varphi\in \mathbb{X},$ it holds that
\begin{equation*}
\begin{split}
&\int_{\Omega}\partial_{t}u(x,t)\varphi\,dx+\int_{\Omega^{2}}|u(x,t)-u(y,t)|^{\wp(x,y)-2} \frac{(u(x,t)-u(y,t))(\varphi(x)-\varphi(y))}{|x-y|^{N+s\wp(x,y)}}\,dx\,dy\\
&=\int_{\Omega}f(x,t)\varphi\,dx \mbox{ for a.a. } t>0.
\end{split}
\end{equation*}

\end{itemize}
\end{definition}

\subsection{Assumptions on weighted function $\mathtt{A}$  }
We introduce the following assumptions for the weight function $\mathtt{A}$:

\noindent \textbf{Assumption} $(W_1).$ We assume that 
\begin{equation*}\label{3.9}
\begin{split}
& \mathtt{A}(x,y,t) = \mathfrak{a}(x,y)\sigma(t), \quad \mathfrak{a} \in L^{\infty}(\Omega^{2}), \; \sigma \in W^{1,2}(0,T), \\
& 0< \mathtt{A}(x,y,t) \leqslant a_{0}  \quad \textrm{for a.e.} \; (x,y) \in \Omega^{2} \; \textrm{and all} \; t \in [0,T].
\end{split}
\end{equation*}


\begin{theorem}\label{teo 3.4}
Assume that $(W_{1})$ is satisfied and $\wp \in (1,+\infty)$. Then for all $f \in L^1(0,T;L^2(\Omega))$ and $u_0 \in L^2(\Omega)$, $(\mathscr{C}_{\varphi_{\wp}^{t},f,u_{0}})$ has a unique weak solution $u_{\wp}$. In particular, if $f \in L^2(0,T;L^2(\Omega))$ and $u_0 \in W^{1,\wp}_{0}(\Omega)$, the weak solution up becomes a strong solution of $(\mathscr{C}_{\varphi_{\wp}^{t},f,u_{0}}).$
\end{theorem}

The second result of this paper  analyzes the asymptotic behavior of $u_{\wp}$ as $\wp \to +\infty$ is given in the following theorem.

\begin{theorem}\label{teo 3.5}
Assume $(W_{1})$ satisfies and define

\begin{equation*}\label{3.17}
\mathtt{K}^{t} := \left\lbrace u \in W^{s,2}_{0}(\Omega): \left|\mathcal{G}_{s}u(x,y) \right| \leqslant \frac{1}{\mathtt{A}(x,y,t)} \; \textrm{for a.e} \; (x,y) \in \Omega^{2} \right\rbrace.
\end{equation*}
Let $(\wp_j)_{j\in \mathbb{N}}$ be a sequence in $(1,+\infty)$ such that $\wp_j \to +\infty$ as $j \to +\infty$. Besides, let $(f_j)_{j\in \mathbb{N}}\subset L^2(0,T;L^2(\Omega))$, $f \in L^2(0,T;L^2(\Omega))$, $(u_{0,j})_{j\in \mathbb{N}}\subset L^2(\Omega)$ and $u_0 \in  \mathtt{K}^0$ such that
\begin{equation*}\label{3.18}
f_j \to f \quad \textrm{ in } \; L^2(0,T;L^2(\Omega)),
\end{equation*}
\begin{equation*}\label{3.19}
u_{0,j} \to u_0 \quad \textrm{ in } \; L^2(\Omega).
\end{equation*}
Then the unique weak solution $u_{j}$ of $(\mathscr{C}_{\varphi_{\wp_{j}}^{t},f_{j},u_{0,j}})$ converges to $u$ as $j \to +\infty$ in the following sense:
\begin{equation*}\label{3.20}
u_{j} \to u \quad \textrm{ strongly in } \; \mathscr{C}([0,T];L^2(\Omega)).
\end{equation*}
Furthermore, the limit $u$ is a unique weak solution of $(\mathscr{C}_{\varphi_{\infty}^{t},f,u_{0}})$, where $\varphi^{t}_{\infty}$ is defined by
\begin{equation*}\label{3.21}
\varphi^{t}_{\infty}(u) := 
\begin{cases}
0 \quad & \textrm{if} \; u \in \mathtt{K}^{t}, \\
+\infty & \textrm{if} \; u \in L^2(\Omega) \setminus \mathtt{K}^{t}.
\end{cases}
\end{equation*}
In particular, if $(1/\wp_j)\displaystyle \int_{\Omega^{2}} \left|\mathcal{G}_{(s,\wp_{j})}u_{0,j}(x) \right|^{\wp_j}\,  dx\,dy$ is bounded as $j \to +\infty$, then the limit $u$ is a strong solution of $(\mathscr{C}_{\varphi_{\infty}^{t},f,u_{0}})$.
\end{theorem}
For periodic problem  \eqref{2.18},  
 first we will find a solution  and  after we will investigate the asymptotic behavior of $u_{\wp}$ as $\wp \to +\infty$. Now we state our third and fourth results.

\begin{theorem}\label{teo 3.8}
Assume $(W_{1})$ and let $p \in [2,+\infty)$. Then for all $f \in L^2(0,T;L^2(\Omega))$, the problem  $(\mathcal{P}_{\varphi^{t}_{\wp},f })$,  has a unique strong solution $u_{\wp}$.
\end{theorem}

\begin{theorem}\label{teo 3.9}
Assume $(W_{1})$ and that $\mathtt{A}(x,y,0) \leqslant \mathtt{A}(x,y,T)$ for a.e. $(x,y) \in \Omega^{2}$. Let $(\wp_j)_{j\in \mathbb{N}}$ be a sequence in $[2,+\infty)$ such that $\wp_j \to +\infty$ as $j \to +\infty$ and let $f_j,f \in L^2(0,T;L^2(\Omega))$ be such that
\begin{equation*}\label{3.42}
f_j \rightharpoonup f \quad \textrm{ in } \; L^2(0,T;L^2(\Omega)).
\end{equation*}
Then a subsequence $(j_k)_{j\in \mathbb{N}}$ of $(j)_{j\in \mathbb{N}}$ can be taken in such a way that the unique strong solution $u_{j_k}$ of $(\mathcal{P}_{\varphi^{t}_{\wp_{j_k}},f_{j_k}})$ converges to $u$ as $k \to \infty$ in the following sense:
\begin{equation*}\label{3.43}
u_{j_k} \to u \textrm{ in } \mathscr{C}([0,T];L^2(\Omega)), \textrm{ weakly in }W^{1,2}(0,T;L^{2}(\Omega)).
\end{equation*}
Furthermore, the limit $u$ is a strong solution of $(\mathcal{P}_{\varphi^{t}_{\infty},f})$.
\end{theorem}

\subsection{Assumptions on the exponents}\label{asumptions}
Now we will establish the hypotheses about the sequence $(\wp_{\jmath}(\cdot,\cdot))_{\jmath\in \mathbb{N}}.$ The sequence $(\wp_{\jmath}(\cdot,\cdot))_{\jmath\in \mathbb{N}}$ is a sequence of class $\mathscr{C}(\overline{\Omega^{2}};(1,\infty))$ satisfying the following assumptions:

\begin{equation}\label{cs1}
\lim_{\jmath \to \infty}\wp^{-}_{\jmath} =\infty \;\; \mbox{and} \;\; \lim_{\jmath \to \infty}(\wp^{+}_{\jmath})^{\frac{1}{\wp^{-}_{\jmath}}}=1,
\end{equation}
where, denote by
\begin{equation*}
\wp^{+}_{\jmath} :=\max_{(x,y) \in \overline{\Omega^{2}}} \wp_{\jmath}(x,y), \;\; \wp^{-}_{\jmath}:= \min_{(x,y) \in \overline{\Omega^{2}}} \wp_{\jmath}(x,y).
\end{equation*}
Finally, for our last result, we will use the notation

\begin{equation}\label{sb1}
\wp_{\jmath}(x,y) =
\begin{cases}
\kappa_{\jmath}(x,y) \to \infty \mbox{ if } (x,y) \in \mathcal{O}^{2}, \\
\kappa(x,y) < \infty \mbox{ if }(x,y) \in (\Omega^{2}) \setminus \overline{\mathcal{O}^{2}} 
\end{cases}
\end{equation}
as $\jmath \to \infty,$ where $\mathcal{O}$ is a nonempty open subset of $\Omega$ such that $|\mathcal{O}^{2}|$, $|(\Omega^{2}) \setminus (\overline{\mathcal{O}^{2}})| > 0$ and $ \kappa : (\Omega^{2}) \setminus (\overline{\mathcal{O}^{2}}) \to (1,\infty)$ and $\kappa_{\jmath} : \mathcal{O}^{2} \to (1,\infty)$ are continuous

\begin{equation*}
\kappa^{+}:= \displaystyle\max_{\displaystyle(x,y) \in (\Omega^{2}) \setminus (\overline{\mathcal{O}^{2}})} \kappa(x,y) \;\; \mbox{  and  } \;\; \kappa^{-}:= \displaystyle\min_{\displaystyle(x,y) \in (\Omega^{2}) \setminus (\overline{\mathcal{O}^{2}})}\kappa(x,y),
\end{equation*}

\begin{equation*}
\kappa^{+}_{\jmath}:=\displaystyle\max_{\displaystyle(x,y) \in \mathcal{O}^{2}} \kappa_{\jmath}(x,y) \;\; \mbox{  and  } \;\; \kappa^{-}_{\jmath}:= \displaystyle\min_{\displaystyle(x,y) \in \mathcal{O}^{2}} \kappa_{\jmath}(x,y)
\end{equation*}
and assume that
\begin{equation}\label{sb2}
1 < \kappa^{-}_{\jmath},\kappa^{-} \;\; \mbox{and} \;\; \kappa^{+}_{\jmath},\kappa^{+} < \infty \;\; \mbox{for all} \;\; \jmath \in \mathbb{N}
\end{equation}
and 
\begin{equation}\label{sb3}
\lim_{\jmath\to \infty}\kappa^{-}_{\jmath}=\infty \;\; \mbox{and} \;\; \lim_{\jmath\to \infty}(\kappa^{+}_{\jmath})^{\frac{1}{\kappa^{-}_{\jmath}}}= 1.
\end{equation}

The fifth result of this paper is given in the following theorem.
\begin{theorem}\label{t1}
Let $( \wp_{\jmath}(\cdot,\cdot))_{\jmath \in \mathbb{N}}$ be the sequence of exponents satisfied the condition \eqref{cs1}. Let  $(f_{\wp_{\jmath}(\cdot)})_{\jmath\in \mathbb{N}}\subset  L^{2} (0,T; L^{2}(\Omega))$ and $(u_{\wp_{\jmath}(\cdot)}^{0})_{\jmath\in \mathbb{N}} \subset L^{2}(\Omega)$ sequences as in problem \eqref{aproximado}  be such that

\begin{equation}\label{cs2}
f_{\wp_{\jmath}(\cdot)} \to f\mbox{  strongly in  } L^{2}(0,T; L^{2}(\Omega)),
\end{equation}

\begin{equation}\label{cs3}
u_{\wp_{\jmath}(\cdot)}^{0} \to u_0 \mbox{  strongly in  } L^{2}(\Omega),
\end{equation}
and let $(u_{\wp_{\jmath}(\cdot)})_{\jmath\in \mathbb{N}}$ be the sequence of solutions of \eqref{aproximado}. Then there exists a function 
$$\mathfrak{u}_{\infty} \in \mathscr{C}([0,T]; L^{2}(\Omega)) \bigcap W^{1,2}_{\mbox{loc}} ((0,T]; L^{2}(\Omega))$$
 such that

\begin{equation*}
\int_{\Omega} \left( f(x,t) - \partial_t \mathfrak{u}_{\infty}(x,t) \right) (v(x) - \mathfrak{u}_{\infty}(x,t))\; dx \leqslant 0 \;\; \mbox{for a.a.} \;\; t \in (0,T),
\end{equation*}
and for all $v \in\mathbb{K}_{\infty}$

\begin{equation*}
\mathfrak{u}_{\infty}(t) \in \mathbb{K}_{\infty} \;\;\mbox{   for a.a.   }t \in (0,T) \;\;\mbox{   and   } \;\;\mathfrak{u}_{\infty}(\cdot,0)=\mathfrak{u}_{\infty}^{0} \mbox{  in } \Omega,
\end{equation*}
 where
\begin{equation*}
\mathbb{K}_{\infty}:= \left\lbrace v\in W^{s,2}_{0}(\Omega) : \sup_{ (x,y) \in \Omega^{2},\, x\neq y  }| \mathcal{G}_{s}v(x,y)| \leqslant 1 \right\rbrace
\end{equation*}
and
\begin{equation}\label{cs4}
u_{\wp_{\jmath}(\cdot)} \to \mathfrak{u}_{\infty} \;\; \mbox{strongly in} \;\; \mathscr{C}([0,T];L^{2}(\Omega)),
\end{equation}

\begin{equation}\label{cs5}
\sqrt{t}\,\, \dfrac{d u_{\wp_{\jmath}(\cdot)}}{dt} \to \sqrt{t}\,\, \dfrac{d\mathfrak{u}_{\infty}}{dt} \;\; \mbox{strongly in} \;\; L^{2}(0,T;L^{2}(\Omega)).
\end{equation}

\noindent Moreover, if 
$$\lim_{\jmath\to \infty}\displaystyle\int_{\Omega^{2}} \frac{1}{\wp_{\jmath}(x,y)} \frac{|u_{\wp_{\jmath}(\cdot)}^{0}(x)-u_{\wp_{\jmath}(\cdot)}^{0}(y)|^{\wp_{\jmath}(x,y)}}{|x-y|^{N+s\wp_{\jmath}(x,y)}}\,dx\,dy= 0,$$ then
\begin{equation}\label{11}
u_{\wp_{\jmath}(\cdot)}\to \mathfrak{u}_{\infty} \mbox{ strongly in }W^{1,2}(0,T;L^{2}(\Omega)).
\end{equation}
\end{theorem}
Now we state our sixth result.
\begin{theorem}\label{cs22} 
Let $( \wp_{\jmath}(\cdot,\cdot))_{\jmath\in \mathbb{N}}$ be the sequence of exponents satisfied the condition \eqref{sb1} such that \eqref{sb2}-\eqref{sb3}
hold. Furthermore, let  $(f_{\wp_{\jmath}(\cdot)} )_{\jmath\in \mathbb{N}} \subset L^{2} (0,T;L^{2}(\Omega))$ and $(u^{0}_{\wp_{\jmath}(\cdot)})_{\jmath\in \mathbb{N}} \subset L^{2}(\Omega)$ as in \eqref{aproximado} be such that \eqref{cs2}-\eqref{cs3} hold and $(u_{\wp_{\jmath}(\cdot)})_{\jmath\in \mathbb{N}}$ be the sequence of solutions of problem \eqref{aproximado}. Then there exists a function $\mathfrak{u}_{\infty} \in \mathscr{C}\left([0,T]; L^{2}(\Omega)\right) \bigcap W^{1,2}_{\mbox{loc}}\left((0,T]; L^{2}(\Omega)\right)$ such that \eqref{cs4}-\eqref{cs5} hold, and moreover, the limit $\mathfrak{u}_{\infty}$ satisfies

\begin{equation}\label{sb4}
\mathfrak{u}_{\infty}(t) \in W^{s,\kappa^{-}}_{0}(\Omega), \; \mathfrak{u}_{\infty}(t) \in  \mathcal{W}^{s,\kappa(\cdot,\cdot)}_{0} (\Omega \setminus \overline{\mathcal{O}}) \;\; \mbox{for a.a.} \;\; t \in (0,T),
\end{equation}
\begin{equation}\label{sb45}
\mathfrak{u}_{\infty}(\cdot,0)=\mathfrak{u}_{\infty}^{0} \mbox{   in   }\Omega
\end{equation}
and solves the following mixed problem for a.a. $t \in (0,T)$: a nonlocal diffusion equation on $\Omega \setminus \overline{\mathcal{O}}$ driven by the fractional $\kappa(\cdot)$-Laplacian,

\begin{equation}\label{sb5}
\partial_{t}\mathfrak{u}_{\infty}(\cdot,t) + (-\Delta)^{s}_{\kappa(\cdot)} \mathfrak{u}_{\infty}(\cdot,t) = f(\cdot,t) \;\; \mbox{in} \;\; \mathscr{D}'(\Omega \setminus \overline{\mathcal{O}})
\end{equation}
 in the sense of distribution and an evolutionary quasivariational inequality over $\mathcal{O},$

\begin{equation}\label{sb6}
\sup_{(x,y)\in \mathcal{O}^{2},\,\,x\neq y}|\mathcal{G}_{s}\mathfrak{u}_{\infty}(x,y)|\leqslant 1,
\end{equation}

\begin{equation}\label{sb7}
\int_{\mathcal{O}} (f(x,t) - \partial_{t}\mathfrak{u}_{\infty}(x,t)) (z(x) - \mathfrak{u}_{\infty}(x,t))\; dx \leqslant 0 \;\; \mbox{for all} \;\; z \in \mathbb{K}_{\infty,\mathcal{O}}(\mathfrak{u}_{\infty}(t)),
\end{equation}
where $\mathbb{K}_{\infty,\mathcal{O}}(\xi)$ is given for each $\xi \in W^{s,\kappa^{-}}_{0}(\Omega)$ by

\begin{equation*}
\mathbb{K}_{\infty,\mathcal{O}}(\xi) := \left\lbrace z \in W^{s,\infty}(\mathcal{O}) : z-\xi \in W^{s,\kappa^{-}}_{0}(\mathcal{O}) \;\; \mbox{and} \;\; \sup_{(x,y)\in \mathcal{O}^{2},\,\,x\neq y}|\mathcal{G}_{s}z(x,y)| \leqslant 1 \right\rbrace.
\end{equation*}
\noindent Besides, if

\begin{equation}\label{sb8}
\lim_{\jmath\to \infty}\int_{\Omega^{2}} \dfrac{1}{\wp_{\jmath}(x,y)} \dfrac{|u^{0}_{\wp_{\jmath}(\cdot)}(x)-u^{0}_{\wp_{\jmath}(\cdot)}(y)|^{\wp_{\jmath}(x,y)}}{|x-y|^{N+s\wp_{\jmath}(x,y)}}\; dxdy =
\int_{\Omega^{2} \setminus (\overline{\mathcal{O}^{2}})} \dfrac{1}{\kappa(x,y)} \dfrac{|u_{0}(x)-u_{0}(y)|^{\kappa(x,y)}}{|x-y|^{N+s\kappa(x,y)}}\; dxdy,
\end{equation}
then \eqref{11} holds, that is,
\begin{equation*}
u_{\wp_{\jmath}(\cdot)} \to \mathfrak{u}_{\infty} \;\; \mbox{strongly in} \;\; W^{1,2}(0,T; L^{2}(\Omega))
\end{equation*}
and

\begin{equation}\label{sb9}
u_{\wp_{\jmath}(\cdot)} \to \mathfrak{u}_{\infty} \;\; \mbox{strongly in} \;\; L^{\sigma}(0,T;\mathcal{W}^{s,\kappa(\cdot,\cdot)}_{0}(\Omega \setminus \overline{\mathcal{O}})) \;\; \mbox{for each} \;\; \sigma \in [1,\infty).
\end{equation}
\end{theorem}
\section{Organization of the Paper}\quad
\noindent\noindent The paper is organized as follows: In Section \ref{preliminaries} we provide the necessary definitions and basic results concerning Fractional Sobolev spaces with variable exponents.    In Section \ref{para} we  deals with the  evolution equation associated the Cauchy abstract problem and  the abstract periodic problem and Mosco convergence. In Section \ref{unicidade} we recall an result of existence solution for problem \eqref{IP}. In Section \ref{prove1} we prove our main results, Theorem \ref{teo 3.4}, Theorem \ref{teo 3.5}, Theorem \ref{teo 3.8}, Theorem \ref{teo 3.9}, Theorem \ref{t1} and Theorem \ref{cs22}.

\section{Preliminary Results}\label{preliminaries}\quad
\noindent \noindent  In this section we collect the necessary information about Sobolev's fractional spaces with variable exponents also called Aronszajn-Gagliardo-Slobodecki spaces. For more details we refer to \cite{lauren,cruz,demengel,diening,valdinoci,elardjh} and the references
cited therein.
\subsection{Nonlocal function spaces}\quad
\noindent\noindent For $s\in(0,1)$ and $\vartheta\in (1,\infty),$ we denote by
$$W^{s,\vartheta}(\Omega):=\Bigg\{u\in L^{\vartheta}(\Omega):\int_{\Omega^{2}}\frac{|u(x)-u(y)|^{\vartheta}}{|x-y|^{N+s\vartheta}}\,dx\,dy<\infty\Bigg\},$$
the fractional Sobolev space endowed with the norm
$$\|u\|_{W^{s,\vartheta}(\Omega)}:=\Bigg(\int_{\Omega}|u|^{\vartheta}\,dx+ \int_{\Omega^{2}}\frac{|u(x)-u(y)|^{\vartheta}}{|x-y|^{N+s\vartheta}}\,dx\,dy  \Bigg)^{\frac{1}{\vartheta}}.$$

Let
 $$W_{0}^{s,\vartheta}(\Omega)=\overline{\mathscr{C}_{0}^{\infty}(\Omega)}^{W^{s,\vartheta}(\Omega)}.$$

For $s\in (0,1),$ $1<\vartheta<\infty$ we have the following Sobolev-Slobodetskii embedding:
\begin{equation*}
 \begin{cases}
 W_{0}^{s,\vartheta}(\Omega)\hookrightarrow L^{\vartheta^{*}}(\Omega),\,\, \vartheta^{*}:=\displaystyle\frac{N\vartheta}{N-s\vartheta}>\vartheta \quad \text{ if } N>s\vartheta, \\
W_{0}^{s,\vartheta}(\Omega)\hookrightarrow L^{\eta}(\Omega) \quad \text{ for all } \eta\in [1,\infty) \text{ if } N=s\vartheta,\\
W_{0}^{s,\vartheta}(\Omega)\hookrightarrow \mathscr{C}^{0,s-\frac{N}{\vartheta}}(\overline{\Omega})\quad \text{ if } N<s\vartheta.
 \end{cases}
 \end{equation*}

\subsection{Nonlocal function spaces with variable exponent}\quad
\noindent\noindent We consider two variable exponents $\eta:\overline{\Omega}\to (1,\infty)$ and $\vartheta:\overline{\Omega}^{2}\to (1,\infty),$ where both $\eta(\cdot)$ and $\vartheta(\cdot,\cdot)$ are continuous functions. We assume that $\vartheta$ is symmetric, that is, $\vartheta(x,y) = \vartheta(y,x)$, and that both $\vartheta$ and $\eta$ are bounded, that is, 

\begin{equation}\label{eq1}
    1<\vartheta^{-}=\min_{(x,y)\in \overline{\Omega^{2}} }\vartheta(x,y)\leqslant \vartheta(x,y)\leqslant \vartheta^{+}=\max_{(x,y)\in \overline{\Omega^{2}}}\vartheta(x,y)<\infty
\end{equation}
and 
\begin{equation}\label{eq2}
    1<\eta^{-}=\min_{x\in \overline{\Omega}}\eta(x)\leqslant \eta(x)\leqslant \eta^{+}=\max_{x\in \overline{\Omega}}\eta(x)<\infty. 
\end{equation}

We define the Banach space $L^{\eta(\cdot)}(\Omega)$ as usual,
\begin{equation*}
 L^{\eta(\cdot)}(\Omega):= \left\{u:\Omega\rightarrow \mathbb{R}:\exists\,\mu>0,\,\text{s.t.} \int_{\Omega} \left|\frac{u(x)}{\mu}\right|^{\eta(x)}\,dx<\infty \right\},
\end{equation*}
with its natural norm
\begin{equation*}
  \| u\|_{ L^{\eta(\cdot)}(\Omega)} := \inf\left\{\mu>0: \int_{\Omega} \left|\frac{u(x)}{\mu}\right|^{\eta(x)}dx\leqslant 1 \right\}.
\end{equation*}
\begin{lemma}\label{holder} The space $L^{\vartheta(\cdot)}(\Omega)$ is a separable, uniformly convex Banach space, and its dual space is $L^{\vartheta'(\cdot)}(\Omega)$ where $\frac{1}{\vartheta(x)}+\frac{1}{\vartheta'(x)}=1.$ For any $u\in L^{\vartheta(\cdot)}(\Omega)$ and $v\in L^{\vartheta'(\cdot)}(\Omega),$ we have
$$\displaystyle\Bigg|\int_{\Omega}uv\,dx \Bigg|\leqslant \Big(\frac{1}{\vartheta^{-}}+\frac{1}{\vartheta'^{-}} \Big)\|u\|_{\vartheta(\cdot)}\|v\|_{\vartheta'(\cdot)}.$$ 
\end{lemma}



Now for $s\in (0,1)$ we define the fractional Sobolev space with variable exponents via the Gagliardo-Slobodetskii approach as follows:

\begin{equation*}
\begin{split}
    \mathcal{W}&=\mathcal{W}^{s,\eta(\cdot),\vartheta(\cdot,\cdot)}(\Omega) \\
    &=\left\{ u\in L^{\eta(\cdot)}(\Omega):
 \int_{\Omega^{2}} \frac{|u(x)-u(y)|^{\vartheta(x,y)}}{\mu^{\vartheta(x,y)}|x-y|^{n+s\vartheta(x,y)}}\,dx\,dy<\infty , \mbox{ for some } \mu >0 \right\}
\end{split}
\end{equation*}
and we set
\begin{equation*}
[u]_{\Omega}^{s,\vartheta(\cdot,\cdot)}:= \inf\left\{ \mu >0:\int_{\Omega^{2}} \frac{|u(x)-u(y)|^{\vartheta(x,y)}}{\mu^{\vartheta(x,y)}|x-y|^{n+s\vartheta(x,y)}}\,dx\,dy\leqslant 1 \right\},
\end{equation*}
\quad\\
the variable exponent Gagliardo-Slobodetskii seminorm. Also, notice that (see for instance \cite{kaufmann}) $\mathcal{W}$ is a Banach space equipped with
the norm
$$ \|u\|_{\mathcal{W}}:=\|u\|_{L^{\eta(\cdot)}(\Omega)}+[u]_{\Omega}^{s,\vartheta(\cdot,\cdot)}.$$

The space $\mathcal{W}_{0}^{s,\eta(\cdot),\vartheta(\cdot,\cdot)}$ denotes the closure of $ \mathscr{C}_{0}^{\infty}(\Omega)$ in $\mathcal{W}.$ Then $\mathcal{W}_{0}^{s,\eta(\cdot),\vartheta(\cdot,\cdot)}$ is a Banach space with the norm
$$\|u\|=[u]^{s,\vartheta(\cdot,\cdot)}_{\Omega}.$$

\begin{theorem}\label{p111}
Let $\Omega\subset\mathbb{R}^{N}$ be bounded domain with smooth boundary and $s \in (0, 1)$. Let $\eta(x)$,
$\vartheta(x, y)$ be continuous variable exponents with $s\,\vartheta(x, y) < N$ for $(x, y)\in \overline{\Omega^{2}}$ and $\eta(x) > \vartheta(x, x)$ for $x\in\overline{\Omega}.$  Then, the embedding $\mathcal{W}^{s,\eta(\cdot),\vartheta(\cdot,\cdot)}(\Omega)\hookrightarrow      L^{\mathrm{r}(\cdot)}(\Omega)$ is continuous and compact for any $\mathrm{r}\in \mathscr{C}(\overline{\Omega}) $ for all $x\in \Omega,$ $\mathrm{r} \in (1, \vartheta^{\star}(x)),$ where
\begin{equation*}
\vartheta^{\star}(x):=\frac{N\vartheta(x,x)}{N-s\vartheta(x,x)}>\mathrm{r}(x)\geqslant \mathrm{r}^{-}>1, \mbox{ for } x\in \overline{\Omega}.
\end{equation*}
 Furthermore, there exists a constant $\mathcal{S} = \mathcal{S}(N, s, \vartheta, \eta, \mathrm{r},\Omega)$ such that for all $u \in \mathcal{W}$,
\begin{equation*}
      \|u\|_{L^{\mathrm{r}(\cdot)}(\Omega)} \leqslant \mathcal{S}\|u\|_{\mathcal{W}}.
\end{equation*}




\end{theorem}
\begin{lemma}
Assume that $\Omega\subset \mathbb{R}^{N}$ is a bounded open domain. Besides, assume that \eqref{eq1} and \eqref{eq2} hold. Then $\mathcal{W}$ is a separable reflexive space. 
\end{lemma}

\begin{remark} 
The $(s,\vartheta(\cdot,\cdot))$-convex modular function $\rho_{s,\vartheta(\cdot,\cdot)}: \mathcal{W}\to \mathbb{R}$ is given by
 
 $$\displaystyle{\rho_{s,\vartheta(\cdot,\cdot)}(u)=\int_{\Omega^{2}}\Bigg(\frac{|u(x)-u(x)|}{|x-y|^{s}} \Bigg)^{\vartheta(x,y)}\frac{dx\,dy}{|x-y|^{N}}}.$$

Note that the space $\mathcal{W}=\mathcal{W}^{s,\eta(\cdot), \vartheta(\cdot,\cdot)}(\Omega)$ can be rewritten as
\begin{equation*}
\mathcal{W}=\Bigg\{u\in L^{\eta(\cdot)}(\Omega):\rho_{s,\vartheta(\cdot,\cdot)}\Big(\frac{u}{\mu} \Big)< \infty, \mbox{ for some } \mu>0 \Bigg\},
\end{equation*}
endowed with the Luxemburg norm
$$\|u\|_{\mathcal{W}}:=\|u\|_{L^{\eta(\cdot)}(\Omega)}+[u]_{\Omega}^{s,\vartheta(\cdot,\cdot)},$$
where
$$[u]_{\Omega}^{s,\vartheta(\cdot,\cdot)}:=\inf\Big\{\mu>0: \rho_{s,\vartheta(\cdot,\cdot)}\Big(\frac{u}{\mu}\Big)\leqslant 1\Big\}.$$
For $0<s<1,$ the term $[\,\cdot\,]_{\Omega}^{s,\wp(\cdot,\cdot)}$ is called the $(s,\wp(\cdot,\cdot))$-Gagliardo-Slobodetskii seminorm.
\end{remark}

The following proposition has an important role in the relationship between the norm $\|\cdot\|_{\mathcal{W}}$   and the $\rho_{s,\vartheta(\cdot,\cdot)}$ convex modular function.

\begin{proposition}\label{lw0}
For $u\in \mathcal{W}$ and $(u_{\iota})_{\iota\in \mathbb{N}}\subset\mathcal{W}$, we have
\begin{itemize}
\item[$\textbf{(1)}$]For $u \in \mathcal{W}\setminus \{0 \}$, $\zeta = \|u\|_{\mathcal{W} } $ if and only if $ \displaystyle\rho_{s,\vartheta(\cdot,\cdot)}\bigg(\displaystyle\frac{u}{\zeta}\bigg)=1.$
   \item[$\textbf{(2)}$]$\|u\|_{  \mathcal{W}}\geqslant 1\Rightarrow \|u\|_{ \mathcal{W}}^{\vartheta^{-}}\leqslant \rho_{s,\vartheta(\cdot,\cdot)}(u)\leqslant \|u\|_{  \mathcal{W} }^{\vartheta^{+}}.$
    \item[$\textbf{(3)}$] $\|u\|_{  \mathcal{W}}\leqslant 1\Rightarrow \|u\|_{  \mathcal{W}}^{\vartheta^{+}}\leqslant \rho_{s,\vartheta(\cdot,\cdot)}(u)\leqslant \|u\|_{  \mathcal{W} }^{\vartheta^{-}}.$

\end{itemize}
\end{proposition}

\begin{remark}
If $\eta(x)=\overline{\vartheta}(x):=\vartheta(x,x),$ we denote $\mathcal{W}^{s,\eta(\cdot),\vartheta(\cdot,\cdot)}(\Omega)$ and $\mathcal{W}_{0}^{s,\eta(\cdot),\vartheta(\cdot,\cdot)}(\Omega)$ by $\mathcal{W}^{s,\vartheta(\cdot,\cdot)}(\Omega)$ and $\mathcal{W}_{0}^{s,\vartheta(\cdot,\cdot)}(\Omega)$ respectively (see \cite{pezzo}). From now on we will denote by $\mathbb{X}=\mathcal{W}_{0}^{s,\vartheta(\cdot,\cdot)}(\Omega).$
\end{remark}

\subsection{Some properties of the fractional $\wp(\cdot)$-Laplacian}\quad
\noindent \noindent Let us now recall some fundamental properties about the  fractional $\wp(\cdot)-$ Laplacian operator. The proof of the results is analogous to \cite{elardjh}.

We consider the operator $\mathcal{A}: \mathbb{X}\to \mathbb{X}'$ defined by
\begin{equation}\label{operadora}
    \langle \mathcal{A}(u),v\rangle_{\mathbb{X},\mathbb{X}'}=\int_{\Omega^{2}}\mathtt{A}(x,y,t)\frac{|u(x)-u(y)|^{\wp(x,y)-2}(u(x)-u(y))(v(x)-v(y))}{|x-y|^{N+s\wp(x,y)}}\,dx\,dy
\end{equation}
for all $u,v\in \mathbb{X}$ and  $\mathbb{X}'$ is the dual space of $\mathbb{X}.$ 
\begin{proposition}\label{monotone}\quad
\begin{itemize}
\item[\textbf{(a)}] The operator $\mathcal{A}:\mathbb{X}\to \mathbb{X}'$ is \textbf{monotone}.

\item[\textbf{(b)}] The operator $\mathcal{A}:\mathbb{X}\to \mathbb{X}'$ is \textbf{coercive}.

\item[\textbf{(c)}] The operator $\mathcal{A}:\mathbb{X}\to \mathbb{X}'$ is \textbf{hemicontinuous}.
\end{itemize}
\end{proposition}

\subsection{Subdifferentials} We present a brief review of some results of the theory of monotone operators. Let $H$ be a Hilbert space with inner product $(\cdot,\cdot)_{H}$ and we consider the functional (also know as potential) $\Psi:H\to (-\infty,\infty)$ be a proper, convex and lower semi-continuous functional with an effective domain
\begin{equation*}
\mathcal{D}(\Psi):=\{u\in H:\Psi(u)<\infty\}.
\end{equation*} 

The subdifferential $\partial\Psi$ of the functional $\Psi$ is defined by
\begin{equation*}
\begin{split}
&\partial \Psi(u):=\{\xi\in H: \forall v\in H:\Psi(v)-\Psi(u)\geqslant (\xi,v-u )_{H}\},\\
&\mathcal{D}(\partial \Psi):=\{u\in \mathcal{D}(\Psi):\exists \,\xi\in H\,\forall v\in H:\Psi(v)-\Psi(u)\geqslant (\xi,v-u )_{H}\}.
\end{split}
\end{equation*}
Note that $\mathcal{D}(\partial\Psi)\subset \mathcal{D}(\Psi).$

\begin{remark}
Taking into account Proposition \ref{monotone}, the operator $\mathcal{A},$ with domain $\mathbb{X},$ is maximal monotone and $R(\mathcal{A}):=\mathcal{A}(\mathbb{X})=\mathbb{X}'$(see \cite[Theorem 1.3, p. 40]{barvu}).
\end{remark}

\begin{remark}\label{remark1}
From \cite[Example 2.3.7, p. 26]{brezis1}, the operator $\mathcal{A}_{\mathcal{H}},$ the realization of $\mathcal{A}$ at $\mathcal{H}=L^{2}(\Omega)$ defined by
\begin{equation*}
\begin{cases} \mathcal{D}(\mathcal{A}_{\mathcal{H}}):=\{u\in \mathbb{X}:\mathcal{A}\in \mathcal{H}\}, & \\  \mathcal{A}_{\mathcal{H}}(u)=\mathcal{A}(u),\,\, \text{if $u\in \mathcal{D}(\mathcal{A}_{\mathcal{H}})$,}
\end{cases}
\end{equation*}
is maximal monotone in $\mathcal{H}.$
\end{remark}

From now on we will write $(-\Delta)_{\wp(\cdot)}^{s}u$ instead of $\mathcal{A}_{\mathcal{H}}(u)$ as we defined above.
\subsection{The weighted fractional $\wp(\cdot)$-Laplacian operator and technical results}\label{sec32}\quad
\noindent\noindent In this subsection, we will see some fundamental properties about the weighted fractional  $\wp(\cdot)$-Laplacian operator. The proof of these results is similar to \cite{elardjh}  and will be omitted. 
Now we will define the functional that for better exposition purposes will only be denoted in this subsection in this way, in each section/subsection the shape of our functional will be specified.

Define the functional $\Phi_{\mathtt{A}}:\mathcal{H}\to \mathbb{R}\cup \{+\infty\}$ by
\begin{equation*} \Phi_{\mathtt{A}}(u) =
\begin{cases}  \displaystyle{\int_{\Omega^{2} }\frac{\mathtt{A}(x,y,t)}{\wp(x,y)}\frac{|u(x)-u(y)|^{\wp(x,y)}}{|x-y|^{N+s\wp(x,y)}}\,dx\,dy}& \text{if $u\in \mathbb{X},$}\\ +\infty & \text{if $u\in \mathcal{H}\setminus  \mathbb{X}$} \end{cases} \end{equation*}
and  $\mathcal{D}(\Phi_{\mathtt{A}})=\{u\in \mathcal{H}:\Phi_{\mathtt{A}}(u)<+\infty\}$ the effective domain of $\Phi_\mathtt{A}.$

Now we will we present some lemmas necessary in the present work.

\begin{lemma}\label{convex}
The functional $\Phi_{\mathtt{A}}$ is convex and proper.
\end{lemma}
\begin{theorem}\label{tu}
The operator  $(-\Delta)_{\wp(\cdot), \mathtt{A}}^{s}$ is the subdiferential $\partial\Phi_{\mathtt{A}}$ of $\Phi_{\mathtt{A}}$.
\end{theorem}

\begin{remark}
Note that $\overline{\mathcal{D}(\partial\Phi_{\mathtt{A}})}^{\mathcal{H}}=\mathcal{D}(\Phi_{\mathtt{A}})$ and $\mathcal{D}(\Phi_{\mathtt{A}})= \mathbb{X}$ (see \cite[Corollary 2.1, p. 55]{barvu}), since the embedding $\mathbb{X}\hookrightarrow \mathcal{H}$ is continuous and compact, we have $\overline{\mathcal{D}(\mathcal{A}_{\mathcal{H}})}^{\mathcal{H}}=\mathcal{H}.$
\end{remark}
\begin{lemma}\label{frechet}
The restriction $\widehat{\Phi}_{\mathtt{A}}=\Phi_{\mathtt{A}}{|_{\mathbb{X}}}$  is of class $\mathscr{C}^{1}( \mathbb{X},\mathbb{R}).$ Moreover, the Fr\'echet derivative $d\widehat{\Phi}_{\mathtt{A}}$ of $\widehat{\Phi_{\mathtt{A}}}$ at $u\in  \mathbb{X}$ coincides with $\mathcal{A}u$ valuated at $u_{|_{\partial\Omega}}=0$ in the sense of distribution, that is,
\begin{equation*}
\langle \widehat{\Phi}'_{\mathtt{A}}(u),v\rangle_{ \mathbb{X}, \mathbb{X}'}=\langle \mathcal{A}(u),v\rangle_{ \mathbb{X}, \mathbb{X}'} \mbox{ for all }u,v \in  \mathbb{X},
\end{equation*}
where the operator $\mathcal{A}$ is defined  in \eqref{operadora}.

 Besides, the restriction $\widehat{\Phi}_{\mathtt{A}}$ is weakly lower semi-continuous in $ \mathbb{X}.$
\end{lemma}

\subsection{Mosco convergence and evolution equation}\quad
\noindent\noindent

 We will make a brief review of the notion of Mosco convergence and recall the convergence result due to Attouch \cite[Theorem 3.19]{attouch2} for evolution equations governed by subdifferential operators. The Mosco convergence is defined as follows:
\begin{definition}\label{M} 
Let $H$ be a Hilbert space and denote by $\Phi(H)$ the set of all proper (i.e., $\varphi \not\equiv \infty$), lower semicontinuous and convex functionals $\varphi:H \to (-\infty,\infty]$. Let $(\varphi_{\jmath})_{\jmath\in \mathbb{N}}$ be a sequence in $\Phi(H)$ and let $\varphi \in \Phi(H)$. Then $\varphi_{\jmath} \to \varphi$ on $H$ in the sense of Mosco as $\jmath \to \infty$ if the following conditions are verifies:
\begin{itemize}
\item[\textbf{(i)}] For each $u \in \mathcal{D}(\varphi)$, there exists a sequence $(u_{\jmath})_{\jmath \in \mathbb{N}}\subset H$ such that $u_{\jmath} \to u$ strongly in $H$ and $\varphi_{\jmath}(u_{\jmath}) \to \varphi(u)$.
\item[\textbf{(ii)}] For each $u\in H$ and any sequence $(u_{\jmath})_{\jmath \in \mathbb{N}}\subset H$ such that $u_{\jmath}\rightharpoonup u$ in $H$. Then 
$$\varphi(u) \leqslant\liminf \limits_{\jmath \to \infty}\varphi_{\jmath}(u_{\jmath}).$$
\end{itemize}

 In \cite{attouch2} the author studies the limit as $\jmath \to \infty$ of solutions for the following Cauchy problems:

\begin{equation}\label{m1}
\dfrac{d u_{\wp_{\jmath}(\cdot)}}{dt}(t) + \partial\varphi_{\jmath} (u_{\wp_{\jmath}(\cdot)}(t)) \ni f_{\jmath}(t) \;\; \mbox{in} \;\; H \;\; \mbox{for} \;\; t \in (0,T),
\end{equation}

\begin{equation}\label{m2}
u_{\wp_{\jmath}(\cdot)}(0)=u^{0}_{\jmath}. 
\end{equation}
\end{definition}

\begin{proposition}\label{p1} (\cite[Theorem 3.74]{attouch2})
Let $\varphi \in \Phi(H)$ and $(\varphi_{\jmath})_{\jmath\in \mathbb{N}}\subset \Phi(H)$ such that $\varphi_{\jmath} \to \varphi$ on $H$ in the sense of Mosco as $\jmath \to \infty$. 
Furthermore, let $f \in L^{2}(0,T,H)$ and $(f_{\jmath})_{\jmath\in \mathbb{N}}\subset L^{2}(0,T,H) $ such that 

\begin{equation*}
f_{\jmath} \to f \;\; \mbox{strongly in} \;\; L^{2}(0,T,H)
\end{equation*}
and let $u^{0}_{\jmath} \in \overline{\mathcal{D}(\varphi_{\jmath})}$ and $u_0 \in \overline{\mathcal{D}(\varphi)}$ such that

\begin{equation*}
u^{0}_{\jmath} \to u_0 \;\; \mbox{strongly in} \;\; H.
\end{equation*}

\noindent Then, the sequence of solutions $(u_{\jmath})_{\jmath\in \mathbb{N}}$ of \eqref{m1}-\eqref{m2} converge to $u$ as $\jmath \to \infty$ in the following sense:

\begin{equation*}
u_{\jmath} \to u \;\; \mbox{strongly in} \;\; \mathscr{C}([0,T];H),
\end{equation*}

\begin{equation*}
\sqrt{t}\,\, \dfrac{d u_{\jmath}}{dt} \to \sqrt{t}\,\, \dfrac{du}{dt} \;\; \mbox{strongly in} \;\; L^{2}(0,T;H).
\end{equation*}
Besides, the limit $u$ is the unique solution of

\begin{equation*}
\dfrac{du}{dt}(t) + \partial\varphi (u(t)) \ni f(t) \;\; \mbox{in} \;\; H, \; t\in (0, T), \; u(0)=u_0.
\end{equation*}
Furthermore, if $\varphi_{\jmath} (u^{0}_{\jmath}) \to \varphi (u_0) < \infty$, then
$$
u_{\jmath} \to u \;\; \mbox{strongly in} \;\; W^{1,2}(0,T;H)\mbox{ and }
\varphi_{\jmath}(u_{\jmath}(\cdot)) \to \varphi (u(\cdot)) \;\; \mbox{uniformly on} \;\; [0,T].
$$
\end{proposition}

\section{Application }\label{para}
In this section, we will talk about evolution equations associated with the Cauchy problem and periodic problems for the constant exponent.

 Consider the following evolution equation \eqref{2.1} in a Hilbert space $H$

\begin{equation}\tag{$\mathcal{E}$}\label{2.1}
\dfrac{du}{dt}(t) + \partial\varphi^{t}u(t) \ni f(t) \quad \textrm{ in } H, \; 0 < t < T,
\end{equation}
where $f \in L^1(0,T;H)$ and $\partial \varphi^{t}$ is the subdifferential of a proper lower semicontinuous convex functional $\varphi^{t} : H \to (-\infty,+\infty]$ for every $t \in [0,T]$.

\begin{definition}\label{def 2.1}
A function $u \in \mathscr{C}([0,T];H)$ is said a strong solution of \eqref{2.1} if verify:
\begin{itemize}
\item[(i)] $u$ is an $H$-valued absolutely continuous function on $[0,T]$.
\item[(ii)] $u(t) \in D(\partial\varphi^{t})$ for a.e. $t \in (0,T)$ and there exists a section $g(t) \in \partial\varphi^{t}(u(t))$ such that
\begin{equation*}\label{2.5}
\dfrac{du}{dt}(t) + g(t) = f(t) \quad \textrm{in} \; H \; \textrm{for a.e.} \; t \in (0,T).
\end{equation*}
\end{itemize}
\end{definition}
Besides, a function $u \in\mathscr{C}([0,T];H)$ is said a weak solution of \eqref{2.1} if there exist sequences $(f_j)_{j\in \mathbb{N}} \subset L^1(0,T;H)$ and $(u_j)_{j\in \mathbb{N}} \subset \mathscr{C}([0,T];H)$ such that $u_j$ is a strong solution of $(\mathcal{E}_{j})$, $f_j \to f$ strongly in $L^1(0,T;H)$ and $u_j \to u$ strongly in $\mathscr{C}([0,T];H)$.

 Hereinafter, we write $(\varphi^{t} )_{t \in [0,T]} \in \Gamma(\alpha,\beta)$ for some functions $\alpha,\beta : [0,+\infty) \times [0,T] \to \mathbb{R}$ if verify:
\begin{itemize}
\item[(i)] $\varphi^{t} \in \Psi(H)$ for all $t \in [0,T]$ (see Definition \ref{M}).
\item[(ii)] There exists $\delta > 0$ such that for all $t_{0} \in [0,T]$ and all $u_{0} \in D(\varphi^{t_0})$, there exists a function $u:J_{\delta}(t_{0})\to H,$ where
$J_{\delta}(t_{0}) := [t_0-\delta,t_0+\delta] \cap [0,T]$ such that for all $t \in J_{\delta}(t_{0})$ and all $\varrho \geqslant |u_{0}|_{H}$,
\begin{equation*}\label{2.8}
\begin{split}
|u(t)-u_{0}|_{H} &\leqslant |\alpha(\varrho,t)-\alpha(\varrho,t_{0})| \sqrt{|\varphi^{t_{0}}(u_{0})|+1 }\\
\varphi^{t}(u(t)) &\leqslant \varphi^{t_{0}}(u_{0}) + |\beta(\varrho,t)-\beta(\varrho,t_{0})| ( |\varphi^{t_{0}}(u_{0})|+1).
\end{split}
\end{equation*}
\end{itemize}

Furthermore, we say $( \varphi^{t} )_{t \in [0,T]} \in \mathcal{B}(\alpha,\beta,C_0,\lbrace \mathfrak{C}_{\varrho} \rbrace_{\varrho \geqslant 0})$ for some functions $\alpha, \beta : [0,+\infty) \times [0,T] \to \mathbb{R}$ and constant $C_0$, $\{ \mathfrak{C}_{\varrho} \}_{\varrho \geqslant 0}$ if the following are all satisfied.
\begin{itemize}
\item[(i)] $(\varphi^t)_{t \in [0,T]} \in \Gamma(\alpha,\beta)$.
\item[(ii)] $\varphi^t(u) \geqslant -C_0 (|u|_H + 1)$ for all $u \in H$ and all $t \in [0,T]$.
\item[(iii)] There exists a function $h : [0,T] \to H$ such that
\begin{equation}\label{2.9}
\sup_{t \in [0,T]} \lbrace |h(t)|_H + |\varphi^t(h(t))| \rbrace + \left( \int_{0}^{T} \left| \dfrac{dh}{dt}(t) \right|^{2}_{H}\; dt \right)^{1/2} \leqslant C_0.
\end{equation}
\item[(iv)] For every $\varrho \in [0,+\infty)$, it follows that
\begin{equation}\label{2.10}
\int_{0}^{T} |\alpha'(\varrho,t)|^{2}\; dt + \int_{0}^{T} |\beta'(\varrho,t)|\; dt \leqslant \mathfrak{C}_{\varrho},
\end{equation}
where $\alpha'$ and $\beta'$ denote $\partial_{t}\alpha $ and $\partial_{t}\beta$, respectively.
\end{itemize}

Let $( \varphi^{t} )_{t \in [0,T]} \in \Gamma(\alpha,\beta)$ be such that $\alpha(\varrho,\cdot) \in W^{1,2}(0,T)$ and $\beta(\varrho,\cdot) \in W^{1,1}(0,T)$ for all $\varrho \in [0,+\infty)$ and introduce the functional $\Phi^{\Lambda}$ defined on $\mathfrak{H}_{\Lambda} := L^2(0,\Lambda;H)$ for any $\Lambda \in(0,T]$:
\begin{equation}\label{2.11}
\Phi^{\Lambda}(u) := 
\begin{cases}
\displaystyle\int_{0}^{\Lambda} \varphi^{t}(u(t))\; dt \quad & \textrm{ if the function } \; t \longmapsto \varphi^{t}(u(t)) \in L^{1}(0,\Lambda), \\
+\infty & \textrm{otherwise}.
\end{cases}
\end{equation}
Then we have that $\Phi^{\Lambda} \in \Psi(\mathfrak{H}_{\Lambda})$. Moreover, \cite[Proposition 1.1]{ken2} implies that for any $u, f \in \mathfrak{H}_{\Lambda}$,

\begin{equation*}\label{2.12}
f \in \partial_{\mathfrak{H}_{\Lambda}} \Phi^{\Lambda}(u) \Longleftrightarrow f(t) \in \partial \varphi^{t}(u(t)) \; \textrm{for a.e.} \; t \in (0,\Lambda).
\end{equation*}

The following result it is very useful to study the convergence of strong solutions $u_n$ for $(\mathcal{E}_{n})$ as $n \to +\infty$. For its proof, we refer to \cite[Proposition 2.7.1]{ken1}.

\begin{proposition}\label{prop 2.4}
For every $j \in \mathbb{N}$, let $( \varphi^{t}_{j} )_{t \in [0,T]} \in \mathcal{B}(\alpha_j, \beta_j, C_0, (\mathfrak{C}_{\varrho} )_{\varrho \geqslant 0})$ and $( \varphi^{t})_{t \in [0,T]} \in \Gamma(\alpha,\beta)$ be such that $\alpha_n(\varrho,\cdot),\alpha(\varrho,\cdot) \in W^{1,2}(0,T)$ and $\beta_j(\varrho,\cdot),\beta(\varrho,\cdot) \in W^{1,1}(0,T)$ for every $\varrho \in [0,+\infty)$. Assume that $\varphi^{t}_{j} \to \varphi^{t}$ on $H$ in the sense of Mosco for every $t \in [0,T]$ as $j \to +\infty$. Then for any $\Lambda \in (0,T]$, we have that
\begin{enumerate}
\item[$\bullet$] For each $u \in D(\Phi^{\Lambda})$, there exists a sequence $(u_j)_{j \in\mathbb{N}}$ in $\mathfrak{H}_{\Lambda}$ such that $u_j \to u$ strongly in $\mathfrak{H}_{\Lambda}$ and $\Phi^{\Lambda}_{j}(u_j) \to \Phi^{\Lambda}(u)$, where $\Phi^{\Lambda}_{j}$ is defined by \eqref{2.11} with $\varphi^{t}$ substitute by $\varphi^{t}_{j}.$
\item[$\bullet$] Let $(u_j)_{j\in \mathbb{N}}$ be a sequence in $\mathfrak{H}_{\Lambda}$ such that $(u_j)_{j\in \mathbb{N}}$ is bounded in $L^{\infty}(0,\Lambda;H)$ and $u_j(t) \rightharpoonup u(t)$ in $H$ for a.e. $t \in (0,\Lambda)$ as $j \to +\infty$ and let $(k_j)_{j\in \mathbb{N}}$ be a subsequence of $(k)_{k\in \mathbb{N}}$. Then $\Phi^{\Lambda}(u)\leqslant \liminf\limits_{j \to +\infty} \Phi^{\Lambda}_{k_j}(u_{j}).$
\end{enumerate}
\end{proposition}


\subsection{Cauchy Problem.} We consider the following Cauchy problem in a Hilbert space $H$

\begin{equation}\tag{$\mathscr{C}_{\varphi^{t},f,u_{0}}$}\label{2.130} 
\begin{cases}
\dfrac{du}{dt}(t) + \partial \varphi^t(u(t)) \ni f(t)  \quad & \textrm{ in } H, \; 0 < t < T, \\
u(0) = u_0,
\end{cases}
\end{equation}

where $\varphi^{t} \in \Psi(H)$ for all $t \in [0,T]$, $f \in L^1(0,T;H)$ and $u_0 \in H$.

We give a definition of solutions for \eqref{2.130} as follows.

\begin{definition}\label{def 2.5}
A function $u \in \mathscr{C}([0,T];H)$ is said a strong (resp. weak) solution of \eqref{2.130} if $u$ is a strong (resp. weak) solution of \eqref{2.1} such that $u(t) \to u_0$ strongly in $H$ as $t \to 0^{+}$.
\end{definition}

Regarding the existence of solutions for \eqref{2.130}, we have the following result.

\begin{theorem}(\cite[Theorem 2.6]{ken1})\label{teo 2.6}
Let $(\varphi^{t})_{t \in [0,T]} \in \Gamma(\alpha,\beta)$ be such that $\alpha(\varrho,\cdot) \in W^{1,2}(0,T)$ and $\beta(\varrho,\cdot) \in W^{1,1}(0,T)$ for every $\varrho \in [0,+\infty)$. Then for all $f \in L^1(0,T;H)$ and $u_0 \in \overline{D(\varphi^{0})}^{H}$, the problem \eqref{2.130} has a unique weak solution $u$ such that the function $t \mapsto \varphi^{t}(u(t))$ is the integrable on $(0,T)$. In particular, if $f \in L^2(0,T;H)$, then the weak solution $u$ satisfies
\begin{equation*}\label{2.14}
\sqrt{t}\dfrac{du}{dt} \in L^2(0,T;H), \quad \sup_{t \in [0,T]} t \varphi^{t}(u(t)) < +\infty.
\end{equation*}
Besides, if $f \in L^2(0,T;H)$ and $u_0 \in D(\varphi^{0})$, then the unique weak solution $u$ becomes a strong solution of $(\mathscr{C}_{\varphi^{t},f,u_0})$ such that
\begin{equation*}\label{2.15}\dfrac{du}{dt} \in L^2(0,T;H), \quad \sup_{t \in [0,T]} \varphi^{t}(u(t)) < +\infty.
\end{equation*}
\end{theorem}

Taking into account Proposition \ref{prop 2.4}, in \cite{akagi}, the authors proving  the following result on the convergence of solutions $u_j$ for $(\mathscr{C}_{\varphi^{t}_{j},f_{j},u_{0,j}})$ as $j \to +\infty$. Its proof can be found in \cite[Theorem 2.7.1]{ken1}.

\begin{theorem}(\cite[Theorem 2.7]{ken1})\label{teo 2.7}
Under the same hypotheses as in Proposition \ref{prop 2.4}. Let $(f_j)_{j\in \mathbb{N}}$ and $(u_{0,j})_{j\in \mathbb{N}}$ be sequences in $L^2(0,T;H)$ and $\overline{D(\varphi^{0}_{j})}^{H}$, respectively, such that $f_j \to f$ strongly in $L^2(0,T;H)$ and $u_{0,j} \to u_0 \in \overline{D(\varphi^0)}^{H}$ strongly in $H$. Then the unique weak solution $u_n$ of $(\mathscr{C}_{\varphi^{t}_{j},f_j,u_{0,j}})$ converges to $u$ in the following sense:
\begin{equation*}\label{2.16}
u_j \to u \quad \textrm{strongly in} \; \mathscr{C}([0,T];H)
\end{equation*}
and the limit $u$ becomes the unique weak solution of \eqref{2.130}. Furthermore,
\begin{equation*}\label{2.17}
\int_{0}^{T} \varphi^{t}_{j}(u_j(t))\; dt \to  \int_{0}^{T} \varphi^{t}(u(t))\; dt.
\end{equation*}
In particular, if $(\varphi^{0}_{j}(u_{0,j}))_{j\in \mathbb{N}}$ is bounded for all $j \in \mathbb{N}$, then the limit $u$ becomes a strong solution of \eqref{2.130}.
\end{theorem}

\subsection{Periodic problem.} We consider the following periodic problem:

\begin{equation}\tag{$\mathcal{P}_{\varphi^{t},f}$}\label{2.180} 
\begin{cases}
\dfrac{du}{dt}(t) + \partial \varphi^t(u(t)) \ni f(t)  \quad & \textrm{ in } H, \; 0 < t < T, \\
u(0) = u(T),
\end{cases}
\end{equation}



\begin{definition}\label{def 2.8}
A function $u \in \mathscr{C}([0,T];H)$ is said a strong solution of \eqref{2.180} if $u$ is a strong solution of \eqref{2.1} such that $u(0)=u(T)$.
\end{definition}

In order to obtain our results, we define the set
\begin{equation*}\label{2.19}
\Psi_{\Theta}(\alpha,\beta,C_0) := \left\lbrace
\begin{aligned}
( \varphi^{t} )_{t \in [0,T]} \in \Gamma(\alpha,\beta): 
\begin{aligned}
& |u|^{2}_{H} \leqslant C_0(\varphi^{t}(u) + 1)\; \forall u \in D(\varphi^{t}), \; \forall t \in [0,T], \\
& D(\varphi^{T}) \subset D(\varphi^{0})
\end{aligned}
\end{aligned}
\right\rbrace
\end{equation*}
for any positive constant $C_0$. Moreover, we write $( \varphi^{t} )_{t \in [0,T]} \in \mathcal{B}_{\Theta}(\alpha,\beta,C_0,\lbrace M_r \rbrace_{r \geqslant 0})$ if the following hold true.
\begin{itemize}
\item[(i)] $(\varphi^{t})_{t \in [0,T]} \in \Psi_{\Theta}(\alpha,\beta,C_0)$.
\item[(ii)] There exists a function $h : [0,T] \to H$ such that \eqref{2.9} holds and $h(0)=h(T)$.
\item[(iii)] For every $\varrho \in [0,+\infty)$, \eqref{2.10} holds.
\item[(iv)] $\varphi^{0}(u) \leqslant \varphi^{T}(u)$ for all $u \in D(\varphi^{T})$.
\end{itemize}
Then it is not difficult to prove that $\mathcal{B}_{\Theta}(\alpha,\beta,C_0,\lbrace \mathfrak{C}_{\varrho} \rbrace_{\varrho \geqslant 0}) \subset \mathcal{B}(\alpha,\beta,C_0,\lbrace \mathfrak{C}_{\varrho} \rbrace_{\varrho \geqslant 0})$.

The existence of strong solutions for \eqref{2.180} follows from the following result.

\begin{theorem}(\cite[Theorem 2.9]{ken1})\label{teo 2.9}
Let $(\varphi^{t})_{t \in [0,T]} \in \Gamma_{\Theta}(\alpha,\beta,C_0)$ be such that $\alpha(\varrho,\cdot) \in W^{1,2}(0,T)$, $\beta(\varrho,\cdot) \in W^{1,1}(0,T)$ for all $\varrho \in [0,+\infty)$. Then for all $f \in L^2(0,T;H)$, the problem \eqref{2.180} has at least one strong solution $u$ verifying
\begin{equation*}\label{2.22}
\dfrac{du}{dt} \in L^2(0,T;H), \quad \sup_{t \in [0,T]} \varphi^{t}(u(t)) < +\infty.
\end{equation*}
In particular, if $\varphi^{t}$ is strictly convex on $H$ for a.e. $t \in (0,T)$, then every strong solution of \eqref{2.180} is unique.
\end{theorem}

Now we study the convergence of strong solutions $u_j$ for $(\mathcal{P}_{\varphi_{j}^{t},f_{j}})$ when $\varphi^{t}_{j} \to \varphi^{t}$ on $H$ in the sense of Mosco and $f_j \rightharpoonup f$  in $L^2(0,T;H)$. However, any studies similar to Theorem \ref{teo 2.7} have not been done on the periodic problem $(\mathcal{P}(\varphi^{t}_{j},f_j))$ yet, which can happen due to a very peculiar difficulty of the periodic problem.

 More accurately, by virtue of Theorems \ref{teo 2.6} and \ref{teo 2.7}, for any $f \in L^2(0,T;H)$ and $u_0 \in \overline{D(\varphi^{0})}^{H}$, every unique weak solution of \eqref{2.130} becomes the limit of unique weak solutions $u_{j}$ for $(\mathrm{C}(\varphi^{t}_{j},f,u_0))$ as $j \to +\infty$. 
 It should be noted, in general, periodic solutions could not be unique.
So there may be a strong solution $u$ of \eqref{2.180} such that any strong solutions $u_j$ of $(\mathcal{P}_{\varphi_{j}^{t},f_{j}})$ never converge to $u$ as $j \to +\infty$. In the Remark \ref{rem 3.10} we give a counter example.

Therefore, by virtue of the difference exposed above, the strong convergence of solutions $u_j$ for $(\mathcal{P}_{\varphi_{j}^{t},f_{j}})$ in $\mathscr{C}([0,T];H)$ cannot be shown in the same way as in the case of the Cauchy problem (see the proof of \cite[Theorem 2.7.1]{ken1}; thus in order to overcome this difficulty, we present the following hypothesis of compactness of the level set on $( \varphi^{t}_{j})_{j \in \mathbb{N}}$.

\begin{itemize}
\item[($\mathfrak{C}_{1}$)] For every $\mu > 0$ and $t \in [0,T]$, any sequence $(u_j)_{j\in \mathbb{N}}$ in $H$ verifying $\sup_{j \in \mathbb{N}} \{ \varphi^{t}_{j}(u_j)+|u_j|_{H}\} \leqslant \mu$ is precompact in $H$.
\end{itemize}

Then this result can be stated as follows.

\begin{theorem}(\cite{akagi})\label{teo 2.10}
For every $j \in \mathbb{N}$, let $( \varphi^{t}_{j})_{t \in [0,T]} \in \mathcal{B}(\alpha_j,\beta_j,C_0,( \mathfrak{C}_{\varrho\geqslant 0}) ) $ and let $(\varphi^{t}_{j})_{t \in [0,T]} \in \Gamma(\alpha,\beta)$ such that $\alpha_j(\varrho,\cdot),\alpha(\varrho,\cdot) \in W^{1,2}(0,T)$ and $\beta_j(\varrho,\cdot),\beta(\varrho,\cdot) \in W^{1,1}(0,T)$ for every $\varrho \in [0,+\infty)$. Suppose that $\varphi^{t}_{n} \to \varphi^{t}$ on $H$ in the sense of Mosco as $j \to +\infty$ and that $(\mathfrak{C}_{1})$ holds. Besides, let $(f_j)_{j\in \mathcal{N}}$ be a sequence in $L^2(0,T;H)$ such that $f_j \rightharpoonup f$ in $L^2(0,T;H)$ and let $(u_{j})_{j\in \mathbb{N}}$ be a sequence of strong solutions for $(\mathcal{P}_{\varphi_{j}^{t},f_{j}})$. Then there exists a subsequence $(j_{k})_{k\in \mathbb{N}}$ of $(j)_{j\in \mathbb{N}}$ such that $u_{j_{k}}$ converges to $u$ in the following sense:
\begin{equation*}\label{2.21}
u_{j_{k}} \to u \quad \textrm{ strongly in } \; \mathscr{C}([0,T];H), \; \textrm{weakly in} \; W^{1,2}(0,T;H),
\end{equation*}
and the limit $u$ becomes a strong solution of \eqref{2.180}. Furthermore,
\begin{equation*}\label{2.22}
\int_{0}^{T} \varphi^{t}_{j}(u_j(t))\; dt \to \int_{0}^{T} \varphi^{t}(u(t))\; dt.
\end{equation*}
\end{theorem}

\begin{remark}\label{rem 2.11}
\begin{itemize}
\item[(a)] In Theorem \ref{teo 2.10}, the limit $u$ possibly depends on the choice of the subsequence $(j_k)_{k\in \mathbb{N}}$.
\item[(b)] Under assumptions on $(\varphi^{t}_{j})_{t \in [0,T]}$ and $( \varphi^{t})_{t \in [0,T]}$ in Theorem \ref{teo 2.10}, we can verify that $( \varphi^{t})_{t \in [0,T]} \in \Gamma_{\Theta}(\alpha,\beta,C_0)$. Indeed, let $u \in D(\varphi^{T})$. Then we can take a sequence $(u_j)_{j\in \mathbb{N}}$ in $H$ such that $u_j \to u$ strongly in $H$ and $\varphi^{T}_{j}(u_j) \to \varphi^{T}(u)$. Furthermore, from the fact that $\varphi^{0}_{j} \leqslant \varphi^{T}_{j}$, it follows that
\begin{equation*}\label{2.23}
\varphi^{0}_{j}(u_j) \leqslant \varphi^{T}_{j}(u_j) \longrightarrow \varphi^{T}(u).
\end{equation*}
\end{itemize}
Besides, since $\liminf\limits_{j \to +\infty} \varphi^{0}_{j}(u_j) \geqslant \varphi^{0}(u)$, we have $\varphi^{0}(u) \leqslant \varphi^{T}(u)$, which implies $D(\varphi^{T}) \subset D(\varphi^{0})$. Analogously we can also infer that $|u|^{2}_{H} \leqslant  C_0(\varphi^{t}(u)+1)$ for all $u \in D(\varphi^{t})$ and $t \in [0,T]$.
\end{remark}



\section{Existence of solutions for  problem \eqref{IP}}\label{unicidade}

\begin{theorem}(\cite[Theorem 1]{elardjh})\label{brez}
Assume that $\wp^{-}>1.$ For $u_{0}\in \mathcal{H}$ and $f\in L_{loc}^{2}(\mathbb{R}_{0}^{+};\mathcal{H}),$ there exists a unique solution $u$ of \eqref{IP} such that the function $t\mapsto \Phi(u(\cdot,t))$ is absolutely continuous in $\mathbb{R}_{0}^{+}.$ In particular, if $u_{0}\in \mathbb{X},$ then 
\begin{equation*}
u\in W_{loc}^{1,2}(\mathbb{R}_{0}^{+};\mathcal{H})\bigcap C_{\omega}(\mathbb{R}_{0}^{+};\mathbb{X})
\end{equation*} 
and $t\mapsto \Phi(u(\cdot,t))$ is absolutely continuous in $\mathbb{R}_{0}^{+}.$\\
Besides, the unique solution $u$ of problem \eqref{IP} depends continuously on initial data $u_{0}$ and $f$ in the following sense: Let $u_{i},\,i=1,2,$ denote the unique solutions of problem \eqref{IP}, with $u_{0}=u_{0,i}\in L^{2}(\Omega)$ and $f=f_{i}\in L_{loc}^{2}(\mathbb{R}_{0}^{+};\mathcal{H})$ for $i=1,2.$ Then 

\begin{equation*}
\|u_{1}(\cdot,t)-u_{2}(\cdot,t)\|_{L^{2}(\Omega)}\leqslant \|u_{0,1}-u_{0,2}\|_{\mathcal{H}}+\int_{0}^{t}\|f_{1}(\cdot,\tau)-f_{2}(\cdot,\tau)\|_{\mathcal{H}}\,d\tau
\end{equation*} 
for all $t\in \mathbb{R}_{0}^{+}.$
\end{theorem}





\section{Proof of the main results}\label{prove1}

\subsection{Proof of Theorem \ref{teo 3.4} (Existence of solution for ($\mathscr{C}_{\varphi^{t}_{\wp_j},f, u_{0}}$))}\label{sec1}\label{limit}\quad\\

For proof of the Theorems \ref{teo 3.4} - \ref{teo 3.9},   let $\mathcal{H} := L^2(\Omega)$ and define $\varphi^{t}_{\wp} : \mathcal{H} \to [0,+\infty]$ as follows:
\begin{equation*}\label{3.10}
\displaystyle{
\varphi^{t}_{\wp}(u) :=
\begin{cases}
\dfrac{1}{\wp} \displaystyle\int_{\Omega^{2}} \mathtt{A}(x,y,t)|\mathcal{G}_{s,\wp}u(x,y)|^{\wp} \; dx\,dy \quad & \textrm{if} \; u \in W^{s,\wp}_{0}(\Omega), \; \mathtt{A}(\cdot, \cdot,t)\mathcal{G}_{s,\wp}u(x,y) \in L^{\wp}(\Omega^{2}), \\
+\infty & \textrm{otherwise}.
\end{cases}}
\end{equation*}
which can be rewritten as

\begin{equation}\label{moscofunc11}
\begin{split}
\varphi^{t}_{\wp}(u) :=
\begin{cases}
\dfrac{1}{\wp} \displaystyle\int_{\Omega^{2}} \mathtt{A}(x,y,t) |\mathcal{G}_{\alpha}u(x,y)| ^{\wp}\; dx\,dy \quad & \textrm{if} \; u \in W^{s,\wp}_{0}(\Omega), \; \mathtt{A}(\cdot, \cdot,t)\mathcal{G}_{\alpha}u(x,y)\in L^{\wp}(\Omega^{2}), \\
+\infty & \textrm{otherwise}.
\end{cases}
\end{split}
\end{equation}
with $s=\alpha-\frac{N}{\wp},$ $\alpha\in (0,1).$

Furthermore, from $(W_1)$ we infer that
\begin{equation*}\label{3.11}
\varphi^{t}_{\wp} \in \Psi(\mathcal{H}), \quad D(\varphi^{t}_{\wp}) = W^{s,\wp}_{0}(\Omega) \quad \mbox{ for all } t \in [0,T].
\end{equation*}

\begin{proof}
We claimed that $(\varphi^{t}_{\wp})_{t \in [0,T]} \in \Gamma(\alpha_{1},0)$ for some function $\alpha_{1} :
[0,+\infty) \times [0,T] \to \mathbb{R}$. Indeed, let $t_0 \in [0,T]$ and $u_0 \in D(\varphi^{t_0}_{p})$ be fixed and define the function $u : [0,T] \to \mathcal{H}$ as follows:
\begin{equation}\label{3.12}
u(t) := \dfrac{\sigma(t)}{\sigma(t_0)}u_0 \in D(\varphi^{t}_{\wp}) \quad \mbox{ for all } t \in [0,T].
\end{equation}
Since $(W_{1})$ says $\mathtt{A}(x,y,t) = \mathfrak{a}(x,y)\sigma(t)$, it follows that
\begin{equation*}\label{3.13}
\mathcal{G}_{(s,\wp)}u(t) = \dfrac{\sigma(t_0)}{\sigma(t)} \mathcal{G}_{(s,\wp)}u_{0}=\frac{\mathtt{A}(x,y,t_{0})}{\mathtt{A}(x,y,t)}\mathcal{G}_{(s,\wp)}u_{0} \quad \mbox{ for all } t \in [0,T],
\end{equation*}
hence

\begin{equation*}\label{3.14}
\varphi^{t}_{\wp}(u(t)) = \dfrac{1}{\wp} \int_{\Omega^{2}} \left(\mathtt{A}(x,y,t)|\mathcal{G}_{(s,\wp)}u| \right)^{\wp} \; dx\,dy = \dfrac{1}{\wp} \int_{\Omega^{2}} \left(\mathtt{A}(x,y,t_{0})|\mathcal{G}_{(s,\wp)}u_{0}| \right)^{\wp} \; dx\,dy = \varphi^{t_0}_{\wp} (u_0).
\end{equation*}
Also, note that
\begin{equation}\label{3.15}
\begin{split}
|u(t)-u_0|_{\mathcal{H}} & = \left| \dfrac{\sigma(t)}{\sigma(t_0)}-1 \right| |u_0|_{\mathcal{H}} \\
& \leqslant \dfrac{1}{a_0} |\mathfrak{a}|_{L^{\infty}(\Omega^{2})} | \sigma(t)-\sigma(t_0)| |u_0|_{\mathcal{H}} \\
& \leqslant |\alpha_{1}(\varrho,t) - \alpha_{1}(\varrho,t_0)| \quad \mbox{ for all } \varrho \geqslant |u_0|_{\mathcal{H}},
\end{split}
\end{equation}
where $\alpha_1$ is given by
\begin{equation*}\label{3.16}
\alpha_{1}(\varrho,t) = \dfrac{\varrho}{a_0} |\mathfrak{a}|_{L^{\infty}(\Omega^{2})} \sigma(t) \in W^{1,2}(0,T).
\end{equation*}
Thus, $( \varphi^{t}_{\wp})_{t \in [0,T]} \in \Gamma(\alpha_{1},0)$ for every $\wp \in (1,+\infty)$. Then, from Theorem \ref{teo 2.6} to $(\mathscr{C}_{\varphi_{\wp}^{t},f,u_{0}})$, which completes the proof.
\end{proof}
\subsection{Proof of Theorem \ref{teo 3.5} (Asymptotic behavior of the limit problem as $\wp \to \infty$ for  \eqref{2.13})}\label{sec1}\label{limit}\quad\\

In order to prove Theorem \ref{teo 3.5} first, we will need the following result.

\begin{lemma}\label{lemma 3.6}
For each $t \in [0,T]$, we define

\begin{equation*}
\mathtt{K}^{t} := \left\lbrace u \in W^{s,2}_{0}(\Omega): \left|\mathcal{G}_{s}(x,y) \right| \leqslant\frac{1}{ \mathtt{A}(x,y,t)} \; \textrm{for a.e} \; (x,y) \in \Omega^{2} \right\rbrace,
\end{equation*}
we have
\begin{equation*}\label{3.28}
\varphi^{t}_{\wp_j} \to \varphi^{t}_{\infty} \quad \textrm{on $\mathcal{H}$ in the sense of Mosco as} \; \wp_j \longrightarrow +\infty.
\end{equation*}
\end{lemma}

\begin{proof}
Let $t \in [0,T]$ be fixed. We claimed that
\begin{equation}\label{3.29}
\begin{split}
\forall u \in D(\varphi^{t}_{\infty}), \quad \exists (u_j)_{j\in \mathbb{N}} \subset \mathcal{H}; \\
u_j \to u \quad \textrm{strongly in $\mathcal{H}$}, \quad \varphi^{t}_{\wp_j}(u_j) \to \varphi^{t}_{\infty}(u).
\end{split}
\end{equation}
Indeed, let $u \in D(\varphi^{t}_{\infty}) = \mathtt{K}^t$ and $u_j := u$ for all $j \in \mathbb{N}$. Since $\mathtt{K}^t \subset W^{s,\wp_j}_{0}(\Omega)$ for all $j\in \mathbb{N}$, and taking into account \eqref{moscofunc11}, we see that

\begin{equation*}\label{3.30}
\begin{split}
0 \leqslant \liminf_{\wp_{j}\to \infty}\varphi^{t}_{\wp_{j}} (u_j) & = \liminf_{\wp_{j}\to \infty}\dfrac{1}{\wp_{j}} \int_{\Omega^{2}}\mathtt{A}(x,y,t) |\mathcal{G}_{s,\wp}u| ^{\wp_{j}}\, dx\,dy \\
&\leqslant \liminf_{\wp_{j}\to \infty}\frac{a_{0}}{\wp_{j}}\bigg(\int_{\Omega^{2}}\bigg[\sup_{x\neq y,\,\,(x,y)\in \Omega^{2}}\frac{|u(x)-u(y)|}{|x-y|^{\alpha}} \bigg]^{\wp_{j}}\,dx\,dy\bigg)\\
& \leqslant \dfrac{1}{\wp_{j}}|\Omega^{2}| \to 0 = \varphi^{t}_{\infty}(u) \quad \textrm{as} \; \wp_{j} \longrightarrow +\infty.
\end{split}
\end{equation*}
So, it is shown \eqref{3.29}.

Now, let us show that
\begin{equation}\label{3.31}
\begin{split}
\forall (u_j)_{j \in \mathbb{N}} \subset \mathcal{H} \quad \textrm{ such that } \; u_j \rightharpoonup u \; \textrm{ in} \; \mathcal{H}, \mbox{ then } \\
\varphi^{t}_{\infty}(u)\leqslant \liminf_{j \to +\infty} \varphi^{t}_{\wp_{j}} (u_j).
\end{split}
\end{equation}
For the case where $u \in D(\varphi^{t}_{\infty}) = \mathtt{K}^t$, it is clear that $\liminf\limits_{j \to +\infty} \varphi^{t}_{\wp_{j}}(u_j) \geqslant 0 = \varphi^{t}_{\infty}(u)$. In the case $u \notin \mathtt{K}^t$, we will argue by contradiction. Assume that
\begin{equation*}\label{3.32}
\exists (u_j) \subset H; \quad u_j \rightharpoonup u\mbox{ in } \; \mathcal{H}, \quad \liminf_{j \to +\infty} \varphi^{t}_{\wp_{j}}(u_j) \geqslant \varphi^{t}_{\infty}(u) = +\infty.
\end{equation*}
Then, we can take a subsequence $(j_{k})_{k \in\mathbb{N}}$ of $(j)$ such that
\begin{equation*}\label{3.33}
\varphi^{t}_{\wp_{j_{k}}}(u_{j_{k}}) \leqslant \mathfrak{C}\mbox{ for all } k\in \mathbb{N},
\end{equation*}
from where we infer that
\begin{equation}\label{3.34}
 \left[\int_{\Omega^{2}} \mathtt{A}(x,y,t) |\mathcal{G}_{(s,\wp)}u_{j_{k}}(x) | ^{\wp_{j_{k}}} dx\,dy \right]^{1/\wp_{j_{k}}} \\
\leqslant ( \wp_{j_{k}}\varphi^{t}_{\wp_{j_{k}}}(u_{j_{k}}))^{1/\wp_{j_{k}}} \to 1 \mbox{ as }  \jmath \to +\infty.
\end{equation}
For simplicity of notation, we write $\wp$ and $u_\wp$ for $\wp_{j_{k}}$ and $u_{j_{k}}$, respectively. Hence, by $(W_1)$ we have
\begin{equation*}\label{3.35}
\left( \int_{\Omega^{2}}(|\mathcal{G}_{(s,\wp)}u_{\wp}|)^{\wp} dx\,dy \right)^{1/\wp} \leqslant C,
\end{equation*}
which implies
\begin{equation*}\label{3.36}
\begin{split}
\left( \int_{\Omega^{2}} \left| \mathcal{G}_{(s,\wp)}u_{\wp} \right|^{q} dx\,dy \right)^{1/q} & \leqslant \left( \int_{\Omega^{2}} \left| \mathcal{G}_{(s,\wp)}u_{\wp} \right|^{\wp}dx\,dy \right)^{1/\wp} |\Omega^{2}|^{(\wp-q)/(\wp q)} \\
& \leqslant C(|\Omega^{2}|+1)^{1/q} \mbox{ for all } q \in [1,\wp].
\end{split}
\end{equation*}
Hence, for each $q \in (1,+\infty)$, we can take a subsequence $(\wp_q)$ of $(\wp)$ such that
\begin{equation*}\label{3.37}
\mathcal{G}_{(s,\wp)}u_{\wp_q} \rightharpoonup \mathcal{G}_{(s,\wp)}u \textrm{ in }L^q(\Omega^{2}).
\end{equation*}
Note that $u \in W^{s,2}_{0}(\Omega)$. In the rest of this proof, we drop $q$ in $\wp_q$. Furthermore, by $(W_{1})$, we can also derive
\begin{equation}\label{3.38}
\mathtt{A}(\cdot,\cdot,t)\mathcal{G}_{(s,\wp)}u_{\wp}\rightharpoonup \mathtt{A}(\cdot,\cdot,t) \mathcal{G}_{(s,\wp)}u \textrm{ in } L^q(\Omega^{2}).
\end{equation}

Then, from \eqref{3.34} and \eqref{3.38}, we have that
\begin{equation*}\label{3.39}
\begin{split}
\left[ \int_{\Omega^{2}}  \mathtt{A}(x,y,t)\left(\left| \mathcal{G}_{(s,\wp)}u \right|  \right)^{q} dx\,dy \right]^{1/q} & \leqslant \liminf_{\wp \to +\infty} \left[ \int_{\Omega^{2}} \mathtt{A}(x,y,t)\left| \mathcal{G}_{(s,\wp)}u_{\wp} \right| ^{q} dx\,dy \right]^{1/q} \\
& \leqslant \liminf_{\wp \to +\infty} \left[ \int_{\Omega^{2}}  \mathtt{A}(x,y,t)\left| \mathcal{G}_{(s,\wp)}u_{\wp} \right| ^{\wp} dx\,dy \right]^{1/\wp} |\Omega^{2}|^{(\wp - q)/(\wp q)} \\
& \leqslant \lim_{\wp \to +\infty} (\wp C)^{1/\wp} |\Omega^{2}|^{(\wp-q)/(\wp  q)} \\
& = |\Omega^{2}|^{1/q}.
\end{split}
\end{equation*}
Taking $q \to +\infty$, we get

\begin{equation*}\label{3.40}
\left| \mathtt{A}(x,y,t)\mathcal{G}_{(s,\wp)}u \right|_{L^{\infty}(\Omega^{2})} \leqslant 1,
\end{equation*}
which contradicts the fact that $u \notin \mathtt{K}^t$. Hence \eqref{3.31} holds true.
\end{proof}
\begin{proof}(Theorem \ref{teo 3.5})
Due to Theorem \ref{teo 2.7}, suffice to show that
\begin{equation}\label{3.22}
( \varphi^{t}_{\wp_j})_{t \in [0,T]} \in \mathcal{B}(\alpha_1,0,C_0,\lbrace \mathfrak{C}_{\varrho} \rbrace_{\varrho \geqslant 0}) \quad \textrm{for some constants } \; C_0, \; \lbrace \mathfrak{C}_{\varrho} \rbrace_{\varrho \geqslant 0},
\end{equation}
\begin{equation}\label{3.23}
(\varphi^{t}_{\infty})_{t \in [0,T]} \in \Gamma(\alpha_1,0), 
\end{equation}
\begin{equation}\label{3.24}
\varphi^{t}_{\wp_j} \to \varphi^{t}_{\infty} \quad \textrm{on $\mathcal{H}$ in the sense of Mosco as} \; \wp_j \to +\infty.
\end{equation}

We first prove \eqref{3.22}. It has already been proven that $(\varphi^{t}_{\wp} )_{t \in [0,T]} \in \Gamma(\alpha_1,0)$ for all $\wp \in (1,+\infty)$. Besides, it is clear that $\varphi^{t}_{\wp} \geqslant 0$ and that $g \equiv 0$ is such that
\begin{equation*}\label{3.25}
\dfrac{dg}{dt}(t) = 0, \quad \varphi^{t}_{\wp}(g(t))=0 \quad \forall t \in [0,T], \; \forall \wp \in (1,+\infty).
\end{equation*}
Thus, we can take $C_0=0$. Moreover, we have
\begin{equation*}\label{3.26}
\int_{0}^{T} |\alpha'_1(\varrho,t)|^{2}dt = \left( \dfrac{\varrho}{a_0} |\mathfrak{a}|_{L^{\infty}(\Omega^{2})} \right)^{2} \int_{0}^{T} |\sigma'(t)|^{2}dt =: \mathfrak{C}_{\varrho},
\end{equation*}
where we note that $\mathfrak{C}_{\varrho}$ is independent of $\wp$. Thus \eqref{3.22} holds.

Arguing as in the proof of Theorem \ref{teo 3.4}, we can obtain \eqref{3.23}. Indeed, $u(t)$ as in \eqref{3.12} satisfies
\begin{equation*}\label{3.27}
\left| \mathcal{G}_{(s,p)}u_{0}(x,t) \right| = \left| \dfrac{\mathtt{A}(x,y,t_{0})}{\mathtt{A}(x,y,t)} \mathcal{G}_{(s,p)}u_{0}(x) \right| \leqslant \mathtt{A}(x,y,t) \quad \textrm{for a.e.} \; (x,y) \in \Omega^{2} \; \textrm{ and all } \; t \in [0,T],
\end{equation*}
for each $u_0 \in \mathtt{K}^{t_0} $ and $t_{0}\in [0,T]$. Hence, we deduce that $\varphi^{t}_{\infty}(u(t)) = \varphi^{t_0}_{\infty}(u_0) = 0$ for all $t \in [0,T]$ and taking into account \eqref{3.15} it follows that $(\varphi^{t}_{\infty})_{t \in [0,T]} \in \Gamma(\alpha_1,0)$.

From Lemma \ref{lemma 3.6} we infer \eqref{3.24}. Therefore, from \eqref{3.22}, \eqref{3.23} and \eqref{3.24}, by Theorem \ref{teo 2.7} we can get our result.

\end{proof}
\subsection{Proof of Theorem \ref{teo 3.8} (Existence of solution periodic problem $(\mathcal{P}_{\varphi^{t}_{\wp},f}$))}\label{sec1}\label{limit}\quad\\

\begin{proof}
Let us show that $( \varphi^{t}_{\wp} )_{t \in [0,T]} \in \Gamma_{\Theta}(\alpha_1,0,\widetilde{C})$ for some positive constant $\widetilde{C}$. Once that the embedding $W^{s,2}_{0}(\Omega) \hookrightarrow \mathcal{H}$ is continuous, we get

\begin{equation*}\label{3.41}
\begin{split}
|u|^{2}_{\mathcal{H}} & \leqslant C \left| \mathcal{G}_{(s,\wp)}u \right|^{2}_{\mathcal{H}} \\
& \leqslant C\frac{1}{|\mathtt{A}|^{2}_{L^{\infty}(Q)}} \int_{\Omega^{2}}  \mathtt{A}(x,y,t) |\mathcal{G}_{(s,\wp)}u|^{2}\,dx\,dy \\
& \leqslant C\frac{1}{|\mathtt{A}|^{2}_{L^{\infty}(Q)} }\left\lbrace \dfrac{2}{\wp} \int_{\Omega^{2}} \mathtt{A}(x,y,t)|\mathcal{G}_{(s,\wp)}u| ^{\wp} dx\;dy + \dfrac{\wp-2}{\wp}|\Omega^{2}| \right\rbrace \\
& \leqslant 2C\frac{1}{|\mathtt{A}|^{2}_{L^{\infty}(Q)}} \lbrace \varphi^{t}_{\wp}(u)+|\Omega^{2}| \rbrace \quad \mbox{ for all }u \in D(\varphi^{t}_{\wp}), \; \mbox{ for all } \wp \geqslant 2,
\end{split}
\end{equation*}
where $Q := \Omega^{2} \times [0,T]$. Thus, once that $D(\varphi^{t}_{\wp}) = W^{s,\wp}_{0}(\Omega)$ for all $t \in [0,T]$ and $(\varphi^{t}_{\wp})_{t \in [0,T]} \in \Gamma(\alpha_1,0)$ for every $\wp \in (1,+\infty)$, we infer that $( \varphi^{t}_{\wp})_{t \in [0,T]} \in \Gamma_{\Theta}(\alpha_1,0,\widetilde{C})$ for some positive constant $\widetilde{C}$ independent of $\wp$. Then, from Theorem \ref{teo 2.9}, we get the existence of a strong solution $u_\wp$ for $(\mathcal{P}_{\varphi^{t}_{\wp},f})$. Furthermore, since $\varphi^{t}_{\wp}$ is strictly convex on $\mathcal{H}$, the periodic solution $u_\wp$ is unique.
\end{proof}

\subsection{Proof of Theorem \ref{teo 3.9} ( Asymptotic behavior of $u_{\wp}$ as $\wp \to +\infty$ for the periodic problem \eqref{2.18})}\label{sec1}\quad\\
Regarding the asymptotic behavior of $u_\wp$ as $\wp \to +\infty$, we have the following result.

\begin{proof}
We claimed that $(\varphi^{t}_{\wp_j})_{t \in [0,T]} \in \mathcal{B}_{\Theta}(\alpha_1,0,C^{\star},( \mathfrak{C}_{\varrho})_{\varrho \geqslant 0})$ for some constants $( \mathfrak{C}_{\varrho})_{\varrho \geqslant 0}$ independent of $j$. Indeed, we have already shown that $(\varphi^{t}_{\wp})_{t \in [0,T]} \in \Psi_{\Theta}(\alpha_1,0,C^{\star})$, where $C^{\star}>0$ is constant independent of $\wp$. Besides, once that $\mathtt{A}(x,y,T) \geqslant \mathtt{A}(x,y,0)$ for a.e. $(x,y) \in \Omega^{2}$, it is obvious that
\begin{equation*}\label{3.44}
\varphi^{0}_{\wp}(u) \leqslant \varphi^{T}_{\wp}(u)\mbox{ for all } u \in D(\varphi^{T}_{p}), \; \mbox{ for all } \wp \in (1,+\infty).
\end{equation*}
To prove this statement we argue as in the proof of Theorem \ref{teo 3.5}.

Now, we prove that $(\varphi^{t}_{\wp_j})_{j \in \mathbb{N}}$ satisfies $(W_{1})$. Let $\mu > 0$ and $t \in [0,T]$ be fixed and let $(u_j)_{j\in \mathbb{N}}$ be a sequence in $\mathcal{H}$ such that
\begin{equation*}\label{3.45}
\varphi^{t}_{\wp_j}(u_j) + |u_j|_{\mathcal{H}} \leqslant \mu \mbox{ for all } j \in \mathbb{N}.
\end{equation*}
For every $\wp_j \geqslant 2$, we have 
\begin{equation*}\label{3.46}
\begin{split}
\left( \int_{\Omega^{2}} | \mathcal{G}_{(s,\wp_{j})}u_{j}|^{2}\, dx\,dy \right)^{1/2} & \leqslant \frac{1}{|\mathtt{A}(\cdot,\cdot,t)|_{L^{\infty}(\Omega^{2})}} \left(\int_{\Omega^{2}} \mathtt{A}(x,y,t)|\mathcal{G}_{(s,\wp_{j})}u_{j} | ^{\wp_j}\,dx\,dy \right)^{1/\wp_j} |\Omega^{2}|^{(\wp_j -2)/(2\wp_j)} \\
& \leqslant \frac{1}{|\mathtt{A}(\cdot,\cdot,t)|_{L^{\infty}(\Omega^{2})} }(\wp_{j} \mu)^{1/\wp_{j}} |\Omega^{2}|^{(\wp_{j} -2)/(2\wp_{j})} \leqslant \mathfrak{C},
\end{split}
\end{equation*}
where $\mathfrak{C}$ is a constant independent of $j$. Thus, since the embedding $W^{s,2}_{0}(\Omega) \hookrightarrow \mathcal{H}$ is compact, it follows that $(u_j)_{j\in \mathbb{N}}$ is precompact in $\mathcal{H}$, which implies $(W_{1})$ with $\varphi^{t}_{j}$ substitute by $\varphi^{t}_{\wp_j}$. 

Furthermore, from Lemma \ref{lemma 3.6} 
\begin{equation*}\label{3.47}
\varphi^{t}_{\wp_j} \to \varphi^{t}_{\infty}\textrm{ on $\mathcal{H}$ in the sense of Mosco}.
\end{equation*}
Therefore from Theorem \ref{teo 2.10}, we can take a subsequence $(j_k)_{k\in \mathbb{N}}$ of $(j)_{j\in  \mathbb{N}}$ such that the unique strong solution $u_{j_k}$ of $(\mathcal{P}_{\varphi^{t}_{\wp_{j_k}},f_{j_k}})$ satisfies
\begin{equation*}\label{3.48}
u_{j_k} \to u \quad \textrm{ in } \; \mathscr{C}([0,T];H), \; \textrm{weakly in} \; W^{1,2}(0,T;\mathcal{H});
\end{equation*}
furthermore, $u$ is a strong solution of $(\mathcal{P}_{\varphi^{t}_{\infty},f})$.
\end{proof}

\begin{remark}\label{rem 3.10}
As stated in the Theorem \ref{teo 3.8}, $(\mathcal{P}_{\varphi^{t}_{\wp},f})$ has a unique strong solution. On the other hand, $(\mathcal{P}_{\varphi^{t}_{\infty},f})$ may have multiple strong solutions. Indeed, let $\widetilde{t} \in [0,T]$ be a minimizer of $\sigma$, that is, $0 < \sigma(\widetilde{t}) \leqslant \sigma(t)$ for all $t \in [0,T]$. Then $\mathtt{K}^{\widetilde{t}} \subset \mathtt{K}^t$ for all $t \in [0,T]$. Thus, for every $u_0 \in \mathtt{K}^{\widetilde{t}}$, $\partial\varphi^{t}_{\infty}(u_0) \ni 0$ for all $t \in [0,T]$ and $u \equiv u_0$ is a strong solution for $(\mathcal{P}_{\varphi^{t}_{\infty},0})$. Thus since $\mathtt{K}^{\widetilde{t}}$ has infinitely many elements, $(\mathcal{P}_{\varphi^{t}_{\infty},0})$ has infinitely many strong solutions.
\end{remark}

Moreover, since $u_{\wp} \equiv 0$ is a unique strong solution of $(\mathcal{P}_{\varphi^{t}_{\wp},0})$ for all $\wp \in (1,+\infty)$, $u_{\wp}$ never converges to any strong solution $u$ of $(\mathcal{P}_{\varphi^{t}_{\infty},0})$ except $u \equiv 0$ as $\wp \to +\infty$.
\subsection{Proof of Theorem \ref{t1} (Asymptotic behavior of the limit problem as $\wp_{\jmath}(\cdot,\cdot)\to\infty$)}\label{sec1}\label{limit}\quad\\
Here, we shall investigate the limiting behavior as $\wp_{\jmath(\cdot)} \to \infty$ of the solution  $u_{\wp_{\jmath}(\cdot)} = u_{\wp_{\jmath}(x,y)}(x,t)$ for \eqref{aproximado}
with $T>0$ and the sequences $(u^{0}_{\jmath})_{\jmath\in\mathbb{N}}\subset L^{2}(\Omega)$ and $(f_{\jmath})_{\jmath\in\mathbb{N}}\subset L^{2}(0,T;L^{2}(\Omega)),$ satisfying

\begin{equation*}
\begin{cases}{}
& u^{0}_{\wp_{\jmath}(\cdot)}\to  u_0 \;\; \mbox{strongly in} \;\; L^{2}(\Omega),\\
& f_{\wp_{\jmath}(\cdot)} \to f \;\; \mbox{strongly in} \;\; L^{2}(0,T;L^{2}(\Omega)).
\end{cases}
\end{equation*}

It should be noted that in \cite{akagi,mayte}, the authors study the limit of solutions as $\wp(\cdot) \to \infty$ has been studied in the constant/variable exponent case to local/nonlocal problems involving the $\wp(\cdot)$-Laplacian and fractional $\wp$-Laplacian for evolution problems and stationary problems respectively.

 As in \cite{akagi,elardjh}, first we characterize the limit of solutions for \eqref{aproximado} via the notion of Mosco convergence of a sequence $(\Phi_{\jmath})_{\jmath\in \mathbb{N}}$ of functionals associated with fractional $\wp_{\jmath}$-Laplacian

\begin{equation*}
\Phi_{\jmath}(u) : =
\begin{cases}
\displaystyle\int_{\Omega^{2}} \dfrac{1}{\wp_{\jmath}(x,y)} \dfrac{|u(x)-u(y)|^{\wp_{\jmath}(x,y)}}{|x-y|^{N+\wp_{\jmath}(x,y)}}\;dxdy & \;\; \mbox{if} \;\; u \in  \mathbb{X}, \\
+ \infty & \;\; \mbox{if} \;\; u \in \mathcal{H} \setminus  \mathbb{X},
\end{cases}
\end{equation*}
where $ \mathbb{X}=\mathcal{W}_{0}^{s,\wp(\cdot,\cdot)}(\Omega),$ $\mathcal{H}=L^{2}(\Omega)$ and $\mathcal{D}(\Phi_{\jmath}) = \lbrace u \in \mathcal{H} : \Phi_{\jmath}(u) < + \infty \rbrace$ we denote the effective domain of $\Phi$ and by exploiting a general theory for the convergence as $\wp_{\jmath(\cdot)} \to \infty$ of solutions for abstract evolution equations governed by subdifferential operators $\partial\Phi_{\jmath}$ of $\Phi_{\jmath}$(see the results of Section \ref{sec32} and Proposition \ref{monotone}) in a Hilbert space $\mathcal{H}$:
\begin{equation*}
\begin{cases}
\dfrac{d u_{\wp_{\jmath}(\cdot)}}{dt} + \partial\Phi_{\jmath}(u_{\wp_{\jmath}(\cdot)}(t)) \ni  f_{\jmath}(t) \;\; \mbox{in} \;\; \mathcal{H} \;\; \mbox{for} \;\; t \in (0,T),\\
u_{\wp_{\jmath}(\cdot)}(0) = u^{0}_{\jmath}.
\end{cases}
\end{equation*}
\quad\\


\noindent To prove Theorem \ref{t1}, we define the functionals $\Phi_{\jmath} : \mathcal{H} := L^{2}(\Omega) \to [0,\infty]$
by

\begin{equation*}
\begin{split}
\Phi_{\jmath}(\xi)& =
\begin{cases}
\displaystyle\int_{\Omega^{2}}\dfrac{1}{\wp_{\jmath}(x,y)} \dfrac{|\xi(x) - \xi(y)|^{\wp_{\jmath}(x,y)}}{|x-y|^{N+s\wp_{\jmath}(x,y)}}dxdy &  \;\; \mbox{if} \;\; \xi \in X,\\ 
+ \infty &  \;\; \mbox{if} \;\; \xi \in \mathcal{H} \setminus X
\end{cases}
\end{split}
\end{equation*}

which can be rewritten as

\begin{equation}\label{moscofunc}
\begin{split}
\Phi_{\jmath}(\xi)& =
\begin{cases}
\displaystyle\int_{\Omega^{2}}\dfrac{1}{\wp_{\jmath}(x,y)} \dfrac{|\xi(x) - \xi(y)|^{\wp_{\jmath}(x,y)}}{|x-y|^{\alpha \wp_{\jmath}(x,y)}}dxdy &  \;\; \mbox{if} \;\; \xi \in X,\\ 
+ \infty &  \;\; \mbox{if} \;\; \xi \in \mathcal{H} \setminus X,
\end{cases}
\end{split}
\end{equation}
with $s=\alpha-\frac{N}{\wp_{\jmath}(x,y)},$ $\alpha\in (0,1).$

The problem $\eqref{aproximado}$ can be rewritten as a the problem \eqref{m1}-\eqref{m2} with $\Phi_{\jmath}$ in $\mathcal{H} = L^{2}(\Omega)$.

Before prove the first main result, we will prove the Mosco convergence of the functional $\Phi_{\jmath}$ as $\wp_{\jmath}(\cdot,\cdot) \to \infty$
 to a convex function $\Phi_{\infty}$ on $L^{2}(\Omega)$ under appropriate conditions $\wp_{\jmath}(\cdot,\cdot) \to \infty$ over $\Omega^{2}.$


\begin{proposition}\label{p2}
Let $(\wp_{\jmath}(\cdot,\cdot))_{\jmath\in \mathbb{N}}$ as in Theorem \ref{t1}, $\Phi_{\jmath}$ converges to $\Phi_{\infty}$ on $L^{2}(\Omega)$ in the sense of Mosco as $\jmath \to \infty$, where $\Phi_{\infty}$ denotes the indicator function over the subset $\mathbb{K}_{\infty} $ of $L^{2}(\Omega)$, that is, the function $\Phi_{\infty}:L^{2}(\Omega)\to [0,\infty]$ is given by

\begin{equation*}
\Phi_{\infty}(\xi) := 
\begin{cases}
0 \;\; \mbox{if} \;\; \xi \in \mathbb{K}_{\infty},\\
+ \infty \;\; \mbox{otherwise}.
\end{cases}
\end{equation*}
\end{proposition}

\begin{proof}
Let $u \in \mathcal{D}(\Phi_{\infty}) = \mathbb{K}_{\infty}$ be fixed. Defining the sequence $u_{\wp_{\jmath}(\cdot)} \equiv u$ in $\mathcal{D}(\Phi_{\infty}),$ taking into account \eqref{moscofunc} and note that 

\begin{equation*}
\begin{split}
0 \leqslant \lim_{\wp_{\jmath}\to \infty} \Phi_{\jmath}(u) & = \lim_{\wp_{j}\to \infty} \int_{\Omega^{2}}\dfrac{1}{\wp_{\jmath}(x,y)} \dfrac{|u(x) - u(y)|^{\wp_{\jmath}(x,y)}}{|x-y|^{\alpha\wp_{\jmath}(x,y)}}\; dx\,dy \\
&\leqslant\lim_{\wp_{\jmath}\to \infty}  \bigg(\int_{\Omega^{2}}\frac{1}{\wp_{\jmath}^{-}}\bigg[ \sup_{x\neq y,\,\,(x,y)\in\Omega^{2}}\frac{|u(x)-u(y)|}{|x-y|^{\alpha}}   \bigg]^{\wp_{\jmath}(x,y)}\,dx\,dy \bigg) \\
& = 0
\end{split}
\end{equation*}
\noindent as $\jmath \to \infty,$ where we have used the fact that $u_{\wp_{\jmath}(\cdot)} = u \in \mathcal{D}(\Phi_{\jmath})$ for each $\jmath \in \mathbb{N}$. Indeed, since $u \in W^{s,2}_{0}(\Omega)$ and $\mathcal{G}_{s}u(x,y) \in L^{\infty}(\Omega^{2},d\mu)\subset L^{\tau}(\Omega^{2},d\mu)$ for all $\tau\geqslant 1$, thus $u\in W_{0}^{s,\tau}(\Omega)$ for all $\tau\in [1,\infty),$ then $u \in \mathcal{W}^{s,\wp_{\jmath}(\cdot,\cdot)}_{0}(\Omega) = \mathcal{D}(\Phi_{\jmath}).$
Thus $\Phi_{\jmath}(u_{\wp_{\jmath}(\cdot)}) \to \Phi_{\infty}(u)$ as $\jmath \to \infty$. Therefore, we have proved \textbf{(i)} of Definition \ref{M}. 

To prove \textbf{(ii)} of Definition \ref{M}, let the sequence $(u_{\wp_{\jmath}(\cdot)})_{\jmath \in \mathbb{N}}\subset L^{2}(\Omega)$ such that $u_{\wp_{\jmath}(\cdot)} \rightharpoonup u$ in $L^{2}(\Omega)$.
\\

\noindent \textbf{Claim}. We assert that

\begin{equation}\label{13}
\Phi_{\infty}(u) \leqslant \liminf_{\jmath \to \infty} \Phi_{\jmath}(u_{\wp_{\jmath}(\cdot)}).
\end{equation}
\noindent  Indeed, if $\liminf\limits_{\jmath \to \infty}\Phi_{\jmath}(u_{\wp_{\jmath}(\cdot)})=\infty$, it is clear that \eqref{13} is satisfied. Now, we have to prove that, when  $\liminf\limits_{\jmath \to \infty}\Phi_{\jmath}(u_{\wp_{\jmath}(\cdot)})<\infty$, up to a subsequence, we have
\begin{equation*}
\Phi_{\jmath}(u_{\wp_{\jmath}(\cdot)}) \leqslant \mathfrak{C},
\end{equation*}
\noindent  for some positive constant $\mathfrak{C}$ independent of $\jmath.$ 

 Note that

\begin{equation*}
\begin{split}
1 \geqslant \dfrac{\Phi_{\jmath}(u_{\wp_{\jmath}(\cdot)})}{\mathfrak{C}} &= \displaystyle\int_{\Omega^{2}} \left[ \dfrac{1}{(\wp_{\jmath}(x,y)\mathfrak{C})^{\frac{1}{\wp_{\jmath}(x,y)}}} \dfrac{|u_{\wp_{\jmath}(\cdot)}(x)-u_{\wp_{\jmath}(\cdot)}(y)|}{|x-y|^{\frac{N}{\wp_{\jmath}(x,y)}+s}} \right]^{\wp_{\jmath}(x,y)}\; dxdy\\
&\geqslant \int_{\Omega^{2}} \left[ \dfrac{1}{(\wp_{\jmath}^{+}\mathfrak{C})^{\frac{1}{\wp_{\jmath}^{-}}}} \dfrac{|u_{\wp_{\jmath}(\cdot)}(x)-u_{\wp_{\jmath}(\cdot)}(y)|}{|x-y|^{\frac{N}{\wp_{\jmath}(x,y)}+s}} \right]^{\wp_{\jmath}(x,y)}\; dxdy,
\end{split}
\end{equation*}
for $\jmath$ large enough. 
\\

Hence, from Proposition \ref{lw0}, we have
\begin{equation}\label{estgrad}
\left\lVert \dfrac{|u_{\wp_{\jmath}(\cdot)}(x)-u_{\wp_{\jmath}(\cdot)}(y)|}{|x-y|^{\frac{N}{\wp_{\jmath}(x,y)}+s}} \right\rVert_{L^{\wp_{\jmath}(\cdot,\cdot)}(\Omega^{2})} \leqslant (\wp^{+}_{\jmath} \mathfrak{C})^{\frac{1}{\wp^{-}_{\jmath}}}.
\end{equation}

\noindent  So, once that $\wp^{-}_{\jmath} - 1 < \wp_{\jmath}(x,y)$ for a.a. $(x,y) \in \Omega^{2}$ and  from Lemma \ref{holder}, we see that

\begin{equation*}
\begin{split}
&\int_{\Omega^{2}} \left( \dfrac{|u_{\wp_{\jmath}(\cdot)}(x)-u_{\wp_{\jmath}(\cdot)}(y)|}{|x-y|^{\frac{N}{\wp_{\jmath}(x,y)}+s}} \right)^{\wp^{-}_{\jmath}-1}\; dx\,dy 
\\ &\leqslant C \left\lVert \left( \dfrac{|u_{\wp_{\jmath}(\cdot)}(x)-u_{\wp_{\jmath}(\cdot)}(y)|}{|x-y|^{\frac{N}{\wp_{\jmath}(x,y)}+s}} \right)^{\wp^{-}_{\jmath}-1} \right\rVert_{L^{\frac{\wp_{\jmath}(\cdot,\cdot)}{\wp^{-}_{\jmath}-1}}(\Omega^{2})} \lVert 1 \lVert_{L^{\mathfrak{z}_{\jmath}(\cdot,\cdot)}(\Omega^{2})},
\end{split}
\end{equation*}
for some $C>0,$ where $\mathfrak{z}_{\jmath}\in \mathscr{C}(\Omega^{2},(0,1))$ is such that
\begin{equation*}
\dfrac{\wp^{-}_{\jmath}-1}{\wp_{\jmath}(x,y)} + \dfrac{1}{\mathfrak{z}_{\jmath}(x,y)} = 1.
\end{equation*}
\noindent On the other hand, by definition, we have
\begin{equation*}
\begin{split}
&\left\| \left( \dfrac{|u_{\wp_{\jmath}(\cdot)}(x)-u_{\wp_{\jmath}(\cdot)}(y)|}{|x-y|^{\frac{N}{\wp_{\jmath}(x,y)}+s}} \right)^{\wp^{-}_{\jmath}-1} \right\|_{L^{\frac{\wp_{\jmath}(\cdot,\cdot)}{\wp^{-}_{\jmath}-1}}(\Omega^{2})}  =\\
 &=\inf \left\lbrace \mu > 0 : \int_{\Omega^{2}} \left( \dfrac{1}{\mu^{\frac{1}{(\wp^{-}_{\jmath}-1)}}} \dfrac{|u_{\wp_{\jmath}(\cdot)}(x)-u_{\wp_{\jmath}(\cdot)}(y)|}{|x-y|^{\frac{N}{\wp_{\jmath}(x,y)+s}}} \right)^{\wp_{\jmath}(x,y)} \; dxdy \leqslant 1 \right\rbrace \\
& = \left\lVert \dfrac{|u_{\wp_{\jmath}(\cdot)}(x)-u_{\wp_{\jmath}(\cdot)}(y)|}{|x-y|^{\frac{N}{\wp_{\jmath}(x,y)}+s}} \right\rVert^{\wp^{-}_{\jmath}-1}_{L^{\wp_{\jmath}(\cdot,\cdot)}(\Omega^{2})}.
\end{split}
\end{equation*}
Furthermore, once that $\mathfrak{z}_{\jmath}(\cdot,\cdot) = \frac{\wp_{\jmath}(\cdot,\cdot)}{\wp_{\jmath}(\cdot,\cdot) - \wp^{-}_{\jmath} + 1}>1$, from \cite[Lema 3.2.11]{diening}, we see that $\lVert 1 \rVert_{L^{\mathfrak{z}_{\jmath}(\cdot,\cdot)}(\Omega)} \leqslant \max \lbrace 1, |\Omega^{2}| \rbrace$. 
\noindent Hence, from Lemma \ref{holder} and Proposition \ref{lw0} and \eqref{estgrad}, we see that
\begin{equation}\label{14}
\begin{split}
\displaystyle\int_{\Omega^{2}}\left( \dfrac{|u_{\wp_{\jmath}(\cdot)}(x)-u_{\wp_{\jmath}(\cdot)}(y)|}{|x-y|^{\frac{N}{\wp_{\jmath}(x,y)}+s}} \right)^{\wp^{-}_{\jmath}-1} dxdy & \leqslant \mathfrak{C} \max \lbrace 1, |\Omega^{2}| \rbrace \left\lVert \dfrac{u_{\wp_{\jmath}(\cdot)}(x)-u_{\wp_{\jmath}(\cdot)}(y)}{|x-y|^{\frac{N}{\wp_{\jmath}(x,y)}+s}} \right\rVert^{\wp^{-}_{\jmath}-1}_{L^{\wp_{\jmath}(\cdot,\cdot)}(\Omega^{2})} \\
& \leqslant \mathfrak{C} \max \lbrace 1, |\Omega^{2}| \rbrace (\wp^{+}_{\jmath}C)^{\frac{(\wp^{-}_{\jmath}-1)}{\wp^{-}_{\jmath}}}.
\end{split}
\end{equation}

\noindent On the other hand, from H\"{o}lder inequality, we have
\begin{equation}\label{g1}
\begin{split}
&\Bigg(\int_{\Omega^{2}} \bigg[\dfrac{|u_{\wp_{\jmath}(\cdot)}(x)-u_{\wp_{\jmath}(\cdot)}(y)|}{|x-y|^{\beta}}\bigg]^{\theta}\,dx\,dy\Bigg)^{\frac{1}{\theta}}\\& \leqslant    (|\Omega^{2} | + 1)^{\frac{1}{\theta}}  \Bigg(\int_{\Omega^{2}} \bigg[\dfrac{|u_{\wp_{\jmath}(\cdot)}(x)-u_{\wp_{\jmath}(\cdot)}(y)|}{|x-y|^{\beta}}\bigg]^{\wp^{-}_{\jmath}-1}\,dx\,dy\Bigg)^{\frac{1}{\wp^{-}_{\jmath}-1}}
\end{split}
\end{equation}
for any  $\theta \in [1, \wp^{-}_{\jmath}-1 ),$ where $s=\beta-\frac{N}{\theta}$ and $\beta\in (0,1)$. Thus, by \eqref{14} and \eqref{g1}, the sequence $( \mathcal{G}_{(s,\theta)}u_{\wp_{\jmath}(\cdot)})_{\jmath\in \mathbb{N}}$ is bounded in $L^{\theta}(\Omega^{2}),$ for each $\theta > 1$, up to a subsequence, we get
\begin{equation*}
\mathcal{G}_{(s,\theta)}u_{\wp_{\jmath}(\cdot)} \rightharpoonup \mathcal{G}_{(s,\theta)}u \;\; \mbox{in} \;\; L^{\theta}(\Omega^{2}) \;\; \mbox{as} \;\; \jmath \to \infty,
\end{equation*}
and $u \in W^{s,2}_{0}(\Omega)$ (see \cite[Lemma 2.4.14]{diening}).

\noindent  Hence, taking $\wp^{-}_{\jmath} \to \infty,$ from \eqref{g1} and once that $(\wp^{+}_{\jmath})^{\frac{1}{\wp^{-}_{\jmath}}} \to 1$, we have

\begin{equation*}
\begin{split}
\Bigg(\int_{\Omega^{2}} \bigg[\dfrac{|u(x)-u(y)|}{|x-y|^{\beta}}\bigg]^{\theta}\,dx\,dy\Bigg)^{\frac{1}{\theta}}& \leqslant \liminf_{\jmath \to \infty} \Bigg(\int_{\Omega^{2}} \bigg[\dfrac{|u_{\wp_{\jmath}(\cdot)}(x)-u_{\wp_{\jmath}(\cdot)}(y)|}{|x-y|^{\beta}}\bigg]^{\theta}\,dx\,dy\Bigg)^{\frac{1}{\theta}} \\
& \leqslant \lim_{\jmath \to \infty} (|\Omega^{2}| + 1)^{\frac{1}{\theta}} \cdot (C\max \lbrace 1,|\Omega^{2}| \rbrace )^{\frac{1}{\wp^{-}_{\jmath}-1}} \cdot (\wp^{+}_{\jmath}\mathfrak{C})^{\frac{1}{\wp^{-}_{\jmath}}} \\
& = (|\Omega^{2}| + 1)^{\frac{1}{\theta}}.
\end{split}
\end{equation*}

\noindent  Then, taking $\theta \to \infty$ and Fatou's Lemma, we achieved 

\begin{equation*}
\begin{split}
\sup_{(x,y)\in \Omega^{2},\,\,x\neq y}\frac{|u(x)-u(y)|}{|x-y|^{s}}&=\lim_{\theta\to \infty}\Bigg(\int_{\Omega^{2}} \bigg[\dfrac{|u(x)-u(y)|}{|x-y|^{\beta}}\bigg]^{\theta}\,dx\,dy\Bigg)^{\frac{1}{\theta}} \\
&\leqslant \lim_{\theta \to \infty}\liminf_{\jmath\to \infty} \Bigg(\int_{\Omega^{2}} \bigg[\dfrac{|u_{\wp_{\jmath}(\cdot)}(x)-u_{\wp_{\jmath}(\cdot)}(y)|}{|x-y|^{\beta}}\bigg]^{\theta}\,dx\,dy\Bigg)^{\frac{1}{\theta}}
\leqslant 1,
\end{split}
\end{equation*}
thus $u\in \mathbb{K}_{\infty}.$



\noindent Hence, since $\Phi_{\jmath} \geqslant 0$ we have 

\begin{equation*}
\Phi_{\infty}(u) = 0 \leqslant \liminf_{\jmath \to \infty} \Phi_{\jmath}(u_{\wp_{\jmath}(\cdot)}).
\end{equation*}

\noindent  Therefore, we achieved \eqref{13}. Therefore, $\Phi_{\jmath} \to \Phi_{\infty}$ on $L^{2}(\Omega)$ in the sense of Mosco as $\jmath \to \infty$.
\end{proof}


\noindent \textbf{Proof of Theorem \ref{t1}}. \begin{proof} Taking into account the Propositions
 \ref{p1} and \ref{p2}, the solutions $u_{\wp_{\jmath}(\cdot)}$ converge to a limit $\mathfrak{u}_{\infty}$ in $\mathbb{X}$ and it uniquely solves 
 
\begin{equation}\label{ev} 
 \frac{d \mathfrak{u}_{\infty}}{dt}(t) + \partial\Phi_{\infty}(\mathfrak{u}_{\infty}(t)) \ni f(t),\;\; 0 < t < T,\;\; \mathfrak{u}_{\infty}(\cdot,0)=\mathfrak{u}_{\infty}^{0}.
\end{equation}

 \noindent From the definition of subdifferential, the evolution equation \eqref{ev} can be rewritten as a variational inequality:

\begin{equation*}
\left( f(t) - \dfrac{d \mathfrak{u}_{\infty}}{dt}(t), v-\mathfrak{u}_{\infty}(t) \right)_{\mathcal{H}} \leqslant \Phi_{\infty}(v) - \Phi_{\infty}(\mathfrak{u}_{\infty}(t)) = 0 \;\; \mbox{for all} \;\; v \in \mathcal{D}(\Phi_{\infty}) = \mathbb{K}_{\infty}
\end{equation*}
and $\mathfrak{u}_{\infty}(t) \in \mathbb{K}_{\infty}$ for a.a. $t \in (0,T).$
\end{proof}
\subsection{ Proof of Theorem \ref{cs22} (Partial diffusion limit problem)}\label{PDL}\quad

\noindent Now, we discuss the case when $\wp_{\jmath}(\cdot,\cdot)\to \infty$ in a subset of $\Omega^{2}. $ More precisely, we study the following problem when \eqref{sb1} holds. In this situation, the limit problem is described as a mixture of two problems, a nonlinear diffusion equation involving the $\wp(\cdot)$-Laplacian in $(\Omega^{2}) \setminus \overline{\mathcal{O}^{2}}$ and evolutionary quasivariational inequality over $\mathcal{O}^{2}$. It should be noted that the set of quasivariational inequality constraints depends on an unknown function
(e.g. the set $\mathbb{K}_{\infty}$ is independent of $u$ as in the previous section).

For each function $\xi : \Omega \to \mathbb{R}$, we use the same letter $\omega$ for the restriction of $\omega$ onto a subset of $\Omega$ without confusion.

\begin{remark} With the aforementioned, we have the following observations.

\begin{itemize}
\item[\textbf{a)}] Roughly speaking, the constraint set of the evolutionary quasivariational inequality requires all test function $z$ in \eqref{sb7} to coincide with $u(\cdot,t)$ on the boundary $\partial D$ at each times $t$.\\
\item[\textbf{b)}] Since  $\wp(\cdot,\cdot) < \infty$ in $(\Omega^{2}) \setminus (\overline{\mathcal{O}^{2}})$ and $\wp^{-}_{\jmath} \to \infty$, the exponents $\wp_{\jmath}(\cdot,\cdot)$ must be discontinuous on $\partial U$ for $\jmath \in \mathbb{N}$ large enough. Thus we need work in the framework of discontinuous exponents (see Section 5 of \cite[Proposition 4]{elardjh}).
\end{itemize}
\end{remark}

 In order to prove the Theorem \ref{cs22}, we first show the Mosco Convergence of $\Phi_{\jmath}$.

\begin{proposition}\label{p3} 
Assume that \eqref{sb3} holds. Then $\Phi_{\jmath}$ converges in the sense of Mosco on $L^{2}(\Omega)$ to the functional $\Phi_{\mathcal{O}} : L^{2}(\Omega) \to [0,\infty]$ given by

\begin{equation*}
\Phi_{\mathcal{O}}(\xi) := 
\begin{cases}
\displaystyle\int_{\Omega^{2}\setminus \overline{\mathcal{O}^{2} }} \dfrac{1}{\kappa(x,y)} \dfrac{|\xi(x)-\xi(y)|^{\kappa(x,y)}}{|x-y|^{N+s\kappa(x,y)}}\; dx\,dy \;\; \mbox{if} \;\; \xi \in W^{s,\kappa^{-}}_{0}(\Omega),\, \xi \in \mathcal{W}^{s,\kappa(\cdot,\cdot)}_{0}(\Omega \setminus \overline{\mathcal{O}})  \mbox{ and} \;\; \\ \quad \quad \quad \quad \quad \quad \quad \quad \quad \quad \quad \quad\quad\quad\quad\quad\quad\quad\quad\quad\quad   \sup_{(x,y)\in \mathcal{O}^{2}\,\,x\neq y}\frac{\xi(x)-\xi(y)}{|x-y|^{s}} \leqslant 1, \\
\infty, \;\;\quad\quad \quad \quad \quad \quad\quad \quad \quad \quad \quad \quad \quad \quad \quad \quad \quad \quad \quad\mbox{otherwise} \;\;
\end{cases}
\end{equation*}
as $\jmath \to \infty.$
\end{proposition}

\begin{proof}
Let $u \in \mathcal{D}(\Phi_{U})$ be fixed and set $u_{\wp_{\jmath}(\cdot)} = u$ for all $\jmath \in \mathbb{N}$. Then $u \in L^{2}(\Omega)$ and 
\begin{equation*}
\dfrac{u(x)-u(y)}{|x-y|^{\frac{N}{\wp_{\jmath}(x,y)}+s}} \in L^{\wp_{\jmath}(\cdot,\cdot)}(\Omega^{2}).
\end{equation*}
Besides, since $\wp^{-}_{\jmath}=\kappa^{-}$ for any $\jmath \in \mathbb{N}$ large enough, we have $W^{s,\kappa^{-}}_{0}(\Omega) = W^{s,\wp^{-}_{\jmath}}_{0}(\Omega)$.

\noindent Hence $u \in \mathcal{W}^{s,\wp_{\jmath}(\cdot,\cdot)}_{0}(\Omega) = \mathcal{D}(\Phi_{\jmath})$ for large enough $\jmath \in \mathbb{N}$. Now, note that from \eqref{sb3}, we have
\begin{equation*}
\begin{split}
\limsup_{\jmath\to \infty}\int_{\mathcal{O}^{2}} \dfrac{1}{\kappa_{\jmath}(x,y)} \dfrac{|u(x)-u(y)|^{\kappa_{\jmath}(x,y)}}{|x-y|^{N+s\kappa_{\jmath}(x,y)}}\, dx\,dy&=  \limsup_{\jmath\to \infty}\int_{\mathcal{O}^{2}} \dfrac{1}{\kappa_{\jmath}(x,y)} \bigg(\dfrac{|u(x)-u(y)|}{|x-y|^{\gamma}}\bigg)^{\kappa_{\jmath}(x,y)} dxdy \\
&\leqslant \lim_{\jmath\to \infty}\Bigg(\int_{\mathcal{O}^{2}} \dfrac{1}{\kappa_{\jmath}(x,y)}\bigg[\sup_{(x,y)\in \Omega^{2},\,\,x\neq y}\dfrac{|u(x)-u(y)|}{|x-y|^{\gamma}}\bigg]^{\kappa_{\jmath}(x,y)}\,dx\,dy    \Bigg)\\
&\leqslant \lim_{\jmath\to \infty}\dfrac{|\mathcal{O}^{2}|}{\kappa^{-}_{\jmath}}=0,
\end{split}
\end{equation*}
with $s=\gamma-\frac{N}{\kappa_{j}(\cdot, \cdot)}.$

Hence, we have

\begin{equation*}
\begin{split}
\Phi_{\jmath}(u_{\wp_{\jmath}(\cdot)}) &= \int_{\mathcal{O}^{2}}\dfrac{1}{\kappa_{\jmath}(x,y)} \dfrac{|u(x)-u(y)|^{\kappa_{\jmath}(x,y)}}{|x-y|^{N+s\kappa_{\jmath}(x,y)}} dxdy 
+ \int_{(\Omega^{2}) \setminus (\overline{\mathcal{O}^{2}})} \dfrac{1}{\kappa(x,y)} \dfrac{|u(x)-u(y)|^{\kappa(x,y)}}{|x-y|^{N+s \kappa(x,y)}} dxdy \\
& \longrightarrow \int_{(\Omega^{2}) \setminus (\overline{\mathcal{O}^{2}}) } \dfrac{1}{\kappa(x,y)} \dfrac{|u(x)-u(y)|^{\kappa(x,y)}}{|x-y|^{N+s\kappa(x,y)}}\,dx\,dy = \Phi_{\mathcal{O}}(u),
\end{split}
\end{equation*}
Thus, we have proved \textbf{(i)} of Definition \ref{M} for $\Phi_{\jmath}$.

 Now, we will prove \textbf{(ii)} of Definition \ref{M}. Let the sequence $(u_{\wp_{\jmath}(\cdot)})_{\jmath\in \mathbb{N}} \subset \mathcal{D}(\Phi_{\jmath})$ be such that $u_{\wp_{\jmath}(\cdot)} \rightharpoonup u$ in $L^{2}(\Omega)$. It is enough to prove that

\begin{equation*}
\liminf_{\jmath \to \infty} \Phi_{\jmath}(u_{\wp_{\jmath}(\cdot)}) < \infty.
\end{equation*}

\noindent Then, up to a subsequence, we have $\Phi_{\jmath}(u_{\wp_{\jmath}(\cdot)}) \leqslant C,$ for some $C>0$, which gives rise to two observations.

 The first, 

\begin{equation}\label{6.11}
\int_{\Omega^{2} \setminus \overline{\mathcal{O}^{2}}} \dfrac{1}{\kappa(x,y)} \dfrac{|u_{\wp_{\jmath}(\cdot)}(x)-u_{\wp_{\jmath}(\cdot)}(y)|^{\kappa(x,y)}}{|x-y|^{N+s\kappa(x,y)}} dxdy \leqslant C,
\end{equation}
which implies, up to a subsequence,

\begin{equation}\label{weakq}
 \mathcal{G}_{(s,\kappa(\cdot,\cdot))}u_{\wp_{\jmath}(\cdot)} \rightharpoonup \mathcal{G}_{(s,\kappa(\cdot,\cdot))}u \;\; \mbox{in} \;\; L^{\kappa(\cdot,\cdot)}(\Omega^{2} \setminus \overline{\mathcal{O}^{2}}).
\end{equation}

\noindent Hence $ u\in W^{s,\kappa(\cdot,\cdot)}(\Omega \setminus \overline{\mathcal{O}})$ (see \cite[Lemma 2.4.17]{diening}).

 The second one is the following:

\begin{equation}\label{6.12}
\int_{\mathcal{O}^{2}} \dfrac{1}{\kappa_{\jmath}(x,y)} \dfrac{|u_{\wp_{\jmath}(\cdot)}(x)-u_{\wp_{\jmath}(\cdot)}(y)|^{\kappa_{\jmath}(x,y)}}{|x-y|^{N+s\kappa_{\jmath}(x,y)}} dx\,dy \leqslant C
\end{equation}
for some $C>0.$

\noindent Arguing as in the proof of Proposition \ref{p2}, since $\kappa^{-}_{\jmath}-1 > \kappa^{-}$ for  large enough $\jmath \in \mathbb{N}$, we see that

\begin{equation}\label{6.13}
\| \mathcal{G}_{(s,\kappa^{-})}u_{\wp_{\jmath}(\cdot)} \|_{L^{\kappa^{-}}(\mathcal{O}^{2})} \leqslant C.\end{equation}
 Thus, from  \eqref{6.11}, \eqref{6.13}, and once that embedding $\mathbb{X}\hookrightarrow W_{0}^{s,\kappa^{-}}(\Omega)$ is continuous, we see that $( \mathcal{G}_{(s,\kappa^{-})}u_{\wp_{\jmath}(\cdot)})_{\jmath \in \mathbb{N}}$ is bounded in $L^{\kappa^{-}}(\Omega^{2})$. From Poincaré inequality for fractional Sobolev spaces,  the sequence $(u_{\wp_{\jmath}(\cdot)})_{\jmath \in \mathbb{N}}$ is bounded in $W^{s,\kappa^{-}}_{0}(\Omega)$, and therefore, up to a subsequence, $u_{\wp_{\jmath}(\cdot)} \rightharpoonup u$ in $W^{s,\kappa^{-}}_{0}(\Omega)$ and $u \in W^{s,\wp^{-}}_{0}(\Omega)$.

\noindent Besides, using \eqref{6.12} and arguing as in Section \ref{limit}, we have

\begin{equation*}
\sup_{(x,y)\in \mathcal{O}^{2}\,\,x\neq y} \frac{|u(x)-u(y)|}{|x-y|^{s}}\leqslant 1.
\end{equation*}

\noindent Thus, we conclude that $u \in \mathcal{D}(\Phi_{\mathcal{O}})$.

 Since the functional

\begin{equation*}
u \longmapsto \int_{(\Omega^{2}) \setminus \overline{\mathcal{O}^{2}}}\dfrac{1}{\kappa(x,y)} \dfrac{|u(x)-u(y)|^{ \kappa(x,y)}}{|x-y|^{N+s\kappa(x,y)}} dx\,dy,
\end{equation*}
is weakly lower semicontinuous in $L^{2}(\Omega)$, it follows that

\begin{equation*}
\begin{split}
\Phi_{\mathcal{O}}(u) & = \int_{(\Omega^{2}) \setminus \overline{\mathcal{O}^{2}}}\dfrac{1}{\kappa(x,y)} \dfrac{|u(x)-u(y)|^{\kappa(x,y)}}{|x-y|^{N+s\kappa(x,y)}} dx\,dy \\
& \leqslant \liminf_{\jmath \to \infty}\int_{(\Omega^{2}) \setminus \overline{\mathcal{O}^{2}}}\dfrac{1}{\kappa(x,y)} \dfrac{|u_{\wp_{\jmath}(\cdot)}(x)-u_{\wp_{\jmath}(\cdot)}(y)|^{\kappa(x,y)}}{|x-y|^{N+s\kappa(x,y)}} dx\,dy  \leqslant \liminf_{\jmath \to \infty} \Phi_{\jmath}(u_{\wp_{\jmath}(\cdot)}).
\end{split}
\end{equation*}

\noindent Therefore, we have proved \textbf{(ii)} of Definition \ref{M}.
\end{proof}

 We are now ready to prove Theorem \ref{cs22}.

\noindent \textbf{Proof of Theorem \ref{cs22}}.\begin{proof} Arguing as in the proof of Theorem \ref{t1}, is not difficult to prove the convergence of $u_{\wp_{\jmath}(\cdot)}$ and note that the limit $\mathfrak{u}_{\infty}$ uniquely solves

\begin{equation}\label{6.14}
\dfrac{d\mathfrak{u}_{\infty}}{dt}(t) + \partial\Phi_{\mathcal{O}}(\mathfrak{u}_{\infty}(t)) \ni f(t) \;\; \mbox{in} \;\; \mathcal{H} = L^{2}(\Omega), \;\; \mathfrak{u}_{\infty}(0)=\mathfrak{u}_{\infty}^{0}.
\end{equation}
\noindent Applying Proposition \ref{p1} with $\varphi_{\jmath}=\Phi_{\jmath}$ and $\varphi=\Phi_{\mathcal{O}}$. So, the main objective of the is to get a representation of \eqref{6.14}. First let us show that

\begin{equation}\label{6.144}
\varsigma = (-\Delta)^{s}_{\kappa(\cdot)}\xi \;\; \mbox{in} \;\; \mathscr{D}'(\Omega \setminus \overline{\mathcal{O}}) \;\; \mbox{if} \;\; \varsigma \in \partial\Phi_{\mathcal{O}}(\xi).
\end{equation}

 From definition of subdifferentials, we have

\begin{equation}\label{6.15}
\Phi_{\mathcal{O}}(v) - \Phi_{\mathcal{O}}(\xi) \geqslant \int_{\Omega} \varsigma(x) (v(x)-\xi(x))\; dx \;\; \mbox{for all} \;\; v \in \mathcal{D}(\Phi_{\mathcal{O}}).
\end{equation}

\noindent In particular, taking

\begin{equation*}
v(x) =
\begin{cases}
\xi(x) \;\; \mbox{in} \;\; \mathcal{O}, \\
\xi(x) + \mathfrak{h} \Theta(x) \;\; \mbox{in} \;\; \Omega \setminus \overline{\mathcal{O}},
\end{cases}
\end{equation*}
\noindent for $\mathfrak{h} \in \mathbb{R}$ arbitrary  and $\Theta \in  \mathscr{C}^{\infty}_{0}(\Omega \setminus \overline{\mathcal{O}})$. Then $v \in \mathcal{D}(\Phi_{\mathcal{O}})$ and we observe that
\small{
\begin{equation*}
\begin{split}
& \displaystyle\int_{(\Omega^{2}) \setminus \overline{\mathcal{O}^{2}}} \dfrac{1}{\kappa(x,y)} \bigg| \dfrac{(\xi+\mathfrak{h}\Theta)(x) - (\xi+\mathfrak{h}\Theta)(y)}{|x-y|^{N+s\kappa(x,y)}} \bigg|^{\kappa(x,y)}\; dx\,dy \\
& \;\;\; - \int_{(\Omega^{2}) \setminus \overline{\mathcal{O}^{2}}} \dfrac{1}{\kappa(x,y)} \dfrac{|\xi(x)-\xi(y)|^{\kappa(x,y)}}{|x-y|^{N+s\kappa(x,y)}}\; dx\,dy \geqslant \int_{\Omega \setminus \overline{\mathcal{O}}} \varsigma(x)\Theta(x)\; dx.
\end{split}
\end{equation*} }

\noindent Hence,
\small{
\begin{equation*}
 \int_{(\Omega^{2}) \setminus \overline{\mathcal{O}^{2}}} \dfrac{|\xi(x)-\xi(y)|^{\kappa(x,y)-2} (\xi(x)-\xi(y)) (\Theta(x)-\Theta(y))}{|x-y|^{N+s\kappa(x,y)}}\; dxdy = \int_{\Omega \setminus \overline{\mathcal{O}}} \varsigma(x)\Theta(x)\; dx,
\end{equation*}}
for all $\Theta \in \mathscr{C}^{\infty}_{0}(\Omega \setminus \overline{\mathcal{O}})$, and therefore, $\varsigma = (-\Delta)^{s}_{\kappa(\cdot)}\xi$ in $\mathscr{D}'(\Omega \setminus \overline{\mathcal{O}})$.

 We claimed that

\begin{equation}\label{6.16}
\int_{\mathcal{O}} \varsigma(x)(  \mathbf{z}(x)-\xi(x))\; dx \leqslant 0 \;\; \mbox{for all} \;\; \mathbf{z} \in \mathbb{K}_{\infty,\mathcal{O}}(\xi) \;\; \mbox{if} \;\; \varsigma \in \partial \Phi_{\mathcal{O}}(\xi).
\end{equation}

\noindent Indeed, let $\mathbf{z} \in \mathbb{K}_{\infty,\mathcal{O}}(\xi)$ and let

\begin{equation*}
v(x) =
\begin{cases}
\mathbf{z}(x) \;\; \mbox{in} \;\; \mathcal{O}, \\
\xi(x) \;\; \mbox{in} \;\; \Omega \setminus \overline{\mathcal{O}},
\end{cases}
\end{equation*}
\noindent substituting in \eqref{6.15}.

 Note that $v \in \mathcal{D}(\Phi_{\mathcal{O}})$, once that $v \in \mathcal{W}^{s,\kappa(\cdot,\cdot)}_{0}(\Omega \setminus \overline{\mathcal{O}})$ and $\sup_{(x,y)\in \mathcal{O}^{2},\,x\neq y}\frac{|v(x)-v(y)|}{|x-y|^{s}} \leqslant 1,$ the zero extension $\widetilde{\mathbf{z}-\xi}$ of  $\mathbf{z}-\xi\in W^{1,\kappa^{-}}_{0}(\mathcal{O})$ in $\Omega$ belongs to $W^{s,\kappa^{-}}_{0}(\Omega)$ and thus $v = \widetilde{\mathbf{z}-\xi}+\xi\in W^{s,\kappa^{-}}_{0}(\Omega)$. Then, we obtain \eqref{6.16}. Therefore, from \eqref{6.144} and \eqref{6.16}, the problem \eqref{6.14} is rewritten as \eqref{sb4} - \eqref{sb7}.

Finally let us prove that if \eqref{sb8} is valid, then we achieved \eqref{sb9}. Once that $\Phi_{\jmath}(u_{\wp_{\jmath}(\cdot)}(\cdot)) \to \Phi_{\mathcal{O}}(\mathfrak{u}_{\infty}(\cdot))$ uniformly on $[0,T]$ from Proposition \ref{p1}, we have
\small{
\begin{equation*}
\begin{split}
\int_{\Omega^{2} \setminus \overline{\mathcal{O}^{2}}}&\dfrac{1}{\kappa(x,y)} \dfrac{|\mathfrak{u}_{\infty}(x,t)-\mathfrak{u}_{\infty}(y,t)|^{\kappa(x,y)}}{|x-y|^{N+s\kappa(x,y)}}\; dx\,dy  = \Phi_{\mathcal{O}}(\mathfrak{u}_{\infty}(t))  = \lim_{\jmath \to \infty} \Phi_{\jmath}(u_{\wp_{\jmath}(\cdot)}(t)) \\
& \geqslant \limsup_{\jmath \to \infty}\int_{\Omega^{2} \setminus \overline{\mathcal{O}^{2}}}\dfrac{1}{\kappa(x,y)} \dfrac{|u_{\wp_{\jmath}(\cdot)}(x,t)-u_{\wp_{\jmath}(\cdot)}(y,t)|^{\kappa(x,y)}}{|x-y|^{N+s\kappa(x,y)}}\; dx\,dy.
\end{split}
\end{equation*}
}

   Remembering that $\mathcal{G}_{(s,\kappa(\cdot,\cdot))}u_{\wp_{\jmath}(\cdot)} \rightharpoonup \mathcal{G}_{(s,\kappa(\cdot,\cdot))}\mathfrak{u}_{\infty}$ in $L^{\kappa(\cdot,\cdot)}(\Omega^{2} \setminus \overline{\mathcal{O}^{2}})$ (see \eqref{weakq}) and note that

\begin{equation*}
\begin{split}
\liminf_{\jmath \to \infty} &\int_{\Omega^{2} \setminus \overline{\mathcal{O}^{2}}}\dfrac{1}{\kappa(x,y)} \dfrac{|u_{\wp_{\jmath}(\cdot)}(x,t)-u_{\wp_{\jmath}(\cdot)}(y,t)|^{\kappa(x,y)}}{|x-y|^{N+s\kappa(x,y)}}\; dx\,dy\\
 &\geqslant  \int_{\Omega^{2} \setminus \overline{\mathcal{O}^{2}}} \dfrac{1}{\kappa(x,y)} \dfrac{|\mathfrak{u}_{\infty}(x,t)-\mathfrak{u}_{\infty}(y,t)|^{\kappa(x,y)}}{|x-y|^{N+s\kappa(x,y)}}\; dx\,dy.
\end{split}
\end{equation*}
Then, arguing as in the proof of  \cite[Theorem 1]{elardjh}, we achieved that  $u_{\wp_{\jmath}(\cdot)}(t) \to \mathfrak{u}_{\infty}(t)$ strongly in $\mathcal{W}^{s,\kappa(\cdot,\cdot)}_{0}(\Omega \setminus \overline{\mathcal{O}})$ for all $t \in [0,T]$. Since $(\Phi_{\jmath}(u_{\wp_{\jmath}(\cdot)}(t)))_{\jmath\in \mathbb{N}}$ is uniformly bounded on $[0,T]$ for all  $\jmath \in \mathbb{N}$ and taking into account Lebesgue Dominated Convergence Theorem, we have $u_{\wp_{\jmath}(\cdot)}\to \mathfrak{u}_{\infty}$ strongly in $L^{\vartheta}(0,T;\mathcal{W}^{s,\kappa(\cdot,\cdot)}_{0}(\Omega \setminus \overline{\mathcal{O}}))$ for any $\vartheta \in [1,\infty )$.
\end{proof}
\section*{Acknowledgments}\quad
\noindent \noindent This study was financed by the Coordena\c{c}\~ao de Aperfei\c{c}oamento de Pessoal de N\'ivel Superior Brasil (CAPES) Finance Code 001.

\end{document}